\documentclass[a4paper]{article}
\usepackage{latexsym,bm}
\usepackage{amssymb,amsmath,amsthm}
\usepackage[english,german]{babel}
\usepackage[colorlinks=true]{hyperref}
\usepackage{color}

\newtheorem{theorem}{Theorem}[section]
\newtheorem{lemma}[theorem]{Lemma}
\newtheorem{remark}[theorem]{Remark}
\newtheorem{proposition}[theorem]{Propositon}
\newtheorem{definition}[theorem]{Definition}

\numberwithin{equation}{section}
\hypersetup{linkcolor=blue,urlcolor=red,citecolor=red}

\begin{document}
\selectlanguage{english}

\title{On switching properties of time optimal controls for linear ODEs
\thanks{This work was partially supported by
the NNSF of China under grants 11601377, 11901432, 11971022.}}

\author{Shulin Qin\thanks{School of Science, Tianjin University of Commerce, Tianjin 300134, China (\texttt{shulinqin@yeah.net})},\quad Gengsheng Wang\thanks{Center for Applied Mathematics, Tianjin University, Tianjin 300072, China (\texttt{wanggs62@yeah.net})},\quad Huaiqiang Yu\thanks{School of Mathematics, Tianjin University, Tianjin 300354, China (\texttt{huaiqiangyu@tju.edu.cn})}}

\date{}
\maketitle

\begin{abstract}
In this paper, we  present some properties of time optimal controls for linear ODEs with the ball-type control constraint. More precisely, for an optimal control, we build up an upper bound for the number of its switching points;
show that  it jumps from one direction to the reverse direction at each switching point;
give its dynamic behaviour between two consecutive switching points; and study  its switching directions.
\end{abstract}

\vskip 10pt
    \noindent
        \textbf{Keywords.} Time optimal control, Switching problem
\vskip 10pt
    \noindent
        \textbf{2010 AMS Subject Classifications.} 93C15, 49J15

\section{Introduction}\label{introduce}

{\bf Notation.}  Let $\mathbb{N}\triangleq\{0,1,2,\ldots\}$,  $\mathbb{N}^+\triangleq\{1,2,\ldots\}$ and $\mathbb{R}^+\triangleq (0,+\infty)$.
Given a subset $E\subset \mathbb{R}^+$,  write  respectively $\sharp E$ and $|E|$ for its cardinality and measure.
Write
 $B_1(0)$  for the unit closed ball in $\mathbb{R}^m$, centered at
 $0\in\mathbb{R}^m$, write $\partial B_1(0)$ for the boundary of $B_1(0)$.
 Given a nonempty subset $S\subset \mathbb{R}$ and a number $\hat t\in \mathbb{R}$, write $\hat t+S\triangleq
 \{\hat t+s: s\in S\}$.
 Given $A\in\mathbb R^{n\times n}$, $B\in\mathbb R^{n\times m}$ and $j\in\mathbb{N}$, let
 $\mbox{span}\{B,AB,\ldots,A^{j}B\}$ be the linear subspace: $\{\sum_{i=0}^{j}A^{i}Bu_i: u_0,\ldots, u_j\in \mathbb{R}^m\}$.
Denote by
$\mathcal{PC}(\mathbb R^+; E)$ (with $E\subset \mathbb R^m$ a nonempty subset)  the space  of all  piecewise continuous and left continuous functions from  $\mathbb R^+$ to $E$ with {\it finitely many discontinuities}.  By the same way, we can define $\mathcal{PC}(I;E)$, where $I$ is an interval of $\mathbb{R}^+$.

\vskip 5pt
\noindent {\bf Control problem.}
We arbitrarily fix  $A\in\mathbb R^{n\times n}$ and $B\in\mathbb R^{n\times m}\setminus\{0\}$ (with $n,m \in \mathbb N^+$).
For each $x_0\in \mathbb{R}^n$ and each $u\in \mathcal{PC}(\mathbb R^+;\mathbb{R}^m)$, we write $x(\cdot;x_0,u)$ for the solution to the  control system:
\begin{eqnarray}\label{system}
 x'(t) = Ax(t)+Bu(t),\;\;t\geq 0; ~
 x(0) = x_0.
\end{eqnarray}
We take the control constrained set as:  $B_1(0)$.
   Given $x_0\in\mathbb{R}^n\setminus\{0\}$, we consider the time optimal control problem:
    \begin{equation}\label{time-optimal-problem}
  (\mathcal{TP})_{x_0}:\;\;
  T^*_{x_0}\triangleq   \inf\big\{\hat t>0:
  \exists\, u\in \mathcal{PC}(\mathbb R^+;B_1(0))
  ~~\mbox{s.t.}~~
  x(\hat t;x_0,u)=0
  \big\}.
\end{equation}
Several notes on this problem are given in order.
\begin{enumerate}
\item [($\textbf{a}_1$)]
 In the problem $(\mathcal{TP})_{x_0}$, we call $T^*_{x_0}$  the optimal time; call $u\in\mathcal{PC}(\mathbb R^+;B_1(0))$  an admissible control if $x(\hat t;x_0,u)=0$ for some $\hat t>0$;  call $u^*_{x_0}\in\mathcal{PC}(\mathbb R^+;B_1(0))$   an optimal control if $x(T^*_{x_0};x_0,u^*_{x_0})=0$
and $u^*_{x_0}(t)=0$ for each $t\in (T^*_{x_0},+\infty)$.
The effective domain of $u^*_{x_0}$ is $(0, T^*_{x_0}]$.

    \item [($\textbf{a}_2$)]
     Many literatures (for instance, \cite{Bellman, Fatt, Fatt-2011, LaSalle, Lin-Wang, Pontryagin, Sontag, Wang-Zhang}) study the following time optimal control problem, with $L^\infty$ controls:
\begin{eqnarray}\label{extended-optimal-time-problem}
 (\widetilde{\mathcal{TP}})_{x_0}:\;\;
  \widetilde T^*_{x_0}\triangleq   \inf\big\{\hat t>0:
  \exists\, u\in L^\infty(\mathbb R^+;B_1(0))
  ~~\mbox{s.t.}~~
  x(\hat t;x_0,u)=0
  \big\}.
\end{eqnarray}
The optimal time, admissible controls, optimal controls to $(\widetilde{\mathcal{TP}})_{x_0}$
are defined in the similar way.
The problems $(\mathcal{TP})_{x_0}$
and $(\widetilde{\mathcal{TP}})_{x_0}$ are equivalent, i.e., {\it they have the same optimal time, and an optimal control to one problem is that of another problem.} (See Lemma \ref{yu-lemma-10-11-1} in Section \ref{appendix,11-3}.)

\end{enumerate}

\noindent{\bf Assumption.}  The following assumption will be effective throughout the paper:
\begin{itemize}
\item [$(\textbf{H})_{x_0}$]  \emph{The problem $(\mathcal{TP})_{x_0}$ has an admissible control.}
\end{itemize}
Several  notes on the assumption $(\textbf{H})_{x_0}$ are given in order.
\begin{enumerate}
\item [($\textbf{b}_1$)] When $(\textbf{H})_{x_0}$ holds,  $(\mathcal{TP})_{x_0}$ has
a unique optimal  control (see  Proposition \ref{pontryagin-maximum-principle}).
\item[($\textbf{b}_2$)]
We  can always  find  $x_0\in\mathbb{R}^n\setminus\{0\}$ so that
$(\mathcal{TP})_{x_0}$ has an admissible control. More precisely, if we let
$$
\mathcal{A}\triangleq\{x_0\in \mbox{span}\{B,AB,\ldots,A^{n-1}B\}:\lim_{T\to +\infty}N(T,x_0)\leq 1\},
$$
where
$$
N(T;x_0)\triangleq \inf\{\|v\|_{L^\infty(0,T;\mathbb{R}^m)}:x(T;x_0,v)=0\},
$$
then when $x_0\in \mathcal{A}$,  $(\mathcal{TP})_{x_0}$ has an admissible control, while
when $x_0\in \mathbb{R}^n\setminus\mathcal{A}$, $(\mathcal{TP})_{x_0}$ has no
 admissible control. This is a direct consequence of
 \cite[Proposition 12 in Appendix A]{Wang-Zhang} and Lemma \ref{yu-lemma-10-11-1} in Section \ref{appendix,11-3} of the current paper.

\item[($\textbf{b}_3$)]
   The following two statements are equivalent:
\begin{itemize}
\item[($\textbf{b}_{3,1}$)] For each $x_0\in\mathbb{R}^n\setminus\{0\}$, $(\mathcal{TP})_{x_0}$
has an admissible control.
    \item[($\textbf{b}_{3,2}$)] It holds that $\mbox{rank}(B,AB,\cdots,A^{n-1}B)=n$ and $\sigma(A)\subset\{\mu\in\mathbb{C}:\mbox{Re}\;\mu\leq 0\}$, where $\sigma(A)$ is the spectrum of $A$.
\end{itemize}
(See  Lemma \ref{yu-lemma-10-11-1} in Section \ref{appendix,11-3}, \cite[Theorem 3.3, pp. 69]{WWXZ}, as well as \cite[Corollary 3.2, pp. 71]{WWXZ}.)

\end{enumerate}

\noindent{\bf Motivations.}
Let $u^*_{x_0}$ be  the optimal control to the problem $(\mathcal{TP})_{x_0}$.
Associated with $u^*_{x_0}$, we introduce  the following concepts:
\begin{definition}\label{Switching point}
Each discontinuity $\hat t$ of $u^*_{x_0}$ is called a switching point of
$u^*_{x_0}$. (It satisfies $u^*_{x_0}(\hat t)=\displaystyle\lim_{t\rightarrow\hat t^-} u^*_{x_0}(t)
 \neq\displaystyle\lim_{t\rightarrow\hat t^+} u^*_{x_0}(t)$.)
 Write
 \begin{equation}\label{1.7-9-28-b}
 \mathbb{S}_{x_0}\triangleq \{\mbox{All  switching points of}\;\; u^*_{x_0}\;\mbox{over}\; (0,T^*_{x_0})\}.
 \end{equation}
\end{definition}
\begin{definition}
  In the case that  $\mathbb{S}_{x_0}\neq \emptyset$, we write
\begin{equation}\label{yu-9-27-10}
    \mathbb{D}_{x_0}\triangleq \left\{v\in \partial B_1(0): \exists\;\hat{t}\in \mathbb{S}_{x_0}\;\mbox{s.t.}\;\lim_{t\to \hat{t}^-}u^*_{x_0}(t)=v\right\}.
\end{equation}
    Each vector in $\mathbb{D}_{x_0}$ is called a  switching direction of  $u^*_{x_0}$. (Notice that  the optimal control $u^*_{x_0}$ has the bang-bang property: $\|u^*_{x_0}(t)\|_{\mathbb{R}^m}=1$ for each $t\in (0,T^*_{x_0}]$ (see Proposition \eqref{pontryagin-maximum-principle}).)
\end{definition}

Several issues on the optimal control should be important and interesting. The first one is about $\mathbb{S}_{x_0}$:
How many switching points does the optimal control have over  $(0,T^*_{x_0})$? What are the locations of these switching points?
How does the optimal control jump at each switching point? The second one is about the dynamic behaviour of the optimal control:
How does the optimal control move between two consecutive switching points? The third one is about  $\mathbb{D}_{x_0}$:
How many switching directions does the optimal control have over  $(0,T^*_{x_0})$? What are these directions?

To our surprise, we can find very few literatures on the studies of the above problems, and most of them (for instance, \cite{Agrachev-Biolo-2017, Agrachev-Biolo-2018, Agrachev-Sachkov,  Biolo, Poggiolini-2017, Pontryagin, Sontag, Sussmann-1979} and references therein)  focus on the upper-bounds
of the number of switching points. These motivate us to study these problems, though we are not able to
solve the above problems completely.

\vskip 5pt
\noindent{\bf Main results.}
    The first main result (of the paper)  concerns $\mathbb{S}_{x_0}$.
 To state it, we let
\begin{eqnarray}\label{A-spectrum-distance}
\mbox{d}_A\triangleq \min\left\{ \pi/|\mbox{Im}\,\lambda| :  \lambda\in\sigma(A) \right\};
\end{eqnarray}
\begin{equation}\label{yu-9-13-b-2}
    \mbox{q}_{A,B}\triangleq\max\{\mbox{dim}\mathcal{V}_A(b): b\;\mbox{is a column of}\;B\},
\end{equation}
   where
   \begin{equation*}\label{1.8-926}
   \mathcal{V}_A(b)\triangleq \mbox{span}\{b,Ab,\ldots,A^{n-1}b\}.
   \end{equation*}
    In (\ref{A-spectrum-distance}), we agree that $1/0=+\infty$.

\begin{theorem}\label{main-theorem}
 Assume that $(\textbf{H})_{x_0}$ holds.  Then the following two conclusions are true:
\begin{description}
  \item[$(i)$]
 If  $I\subset (0,T^*_{x_0})$ is an open interval  with  $|I|\leq \mbox{d}_A$, then
 $\sharp[\mathbb{S}_{x_0}\cap I]\leq (\mbox{q}_{A,B}-1)$.

  \item[$(ii)$] Let  $u^*_{x_0}$ be the optimal control to $(\mathcal{TP})_{x_0}$.
  Let $\hat t\in\mathbb{S}_{x_0}$.  Then  $\lim_{t\rightarrow\hat t^-} u^*_{x_0}$ and
      $\lim_{t\rightarrow\hat t^+} u^*_{x_0}$ are symmetric with respect to
      $0$ (the center of $B_1(0)$), i.e.,
  \begin{eqnarray}\label{main-theorem-1}
  \lim_{t\rightarrow\hat t^-} u^*_{x_0}(t) + \lim_{t\rightarrow\hat t^+} u^*_{x_0}(t) = 0.
  \end{eqnarray}
\end{description}
\end{theorem}
Several notes on  Theorem \ref{main-theorem} are given in order.
\begin{enumerate}
\item [($\textbf{c}_1$)] In plain language, $(i)$ of  Theorem \ref{main-theorem} says that the optimal control
has at most $(\mbox{q}_{A,B}-1)$ switching points on any open interval $I\subset (0,T^*_{x_0})$ with  $|I|\leq d_A$, while $(ii)$ of  Theorem \ref{main-theorem} says that the optimal control jumps from one direction to its reverse direction at each switching point.
\item [($\textbf{c}_2$)] By $(i)$ of Theorem \ref{main-theorem}, we see that if $\sigma(A)\subset \mathbb{R}$,
then $\sharp\mathbb{S}_{x_0}\leq (\mbox{q}_{A,B}-1)$, ({\it i.e., the optimal control
  to $(\mathcal{TP})_{x_0}$ has at most $(\mbox{q}_{A,B}-1)$ switching points over the interval  $(0,T^*_{x_0})$}). Indeed, in the case that  $\sigma(A)\subset\mathbb{R}$, we see from  (\ref{A-spectrum-distance}) that $\mbox{d}_A=+\infty$.

  It deserves to mention two facts: First,  in \cite[Theorem 10, Chapter III, pp.120]{Pontryagin}, the authors proved that each component of the optimal control has at most $(n-1)$ switching points, when
     $\sigma(A)\subset \mathbb{R}$ and $B_1(0)$ is replaced by a rectangle. (There, the optimal control jumps from one vertex to another at a switching point.) Second, it follows from \eqref{yu-9-13-b-2}
  that $\mbox{q}_{A,B}\leq n$.  But, even when
   $\sigma(A)\subset\mathbb{R}$ and $\mbox{rank}(B,AB,\cdots, A^{n-1}B)=n$, there are some examples to show that $\mbox{q}_{A,B}<n$ (see, for instance, \cite[Example 2.8]{QSL2}).

   \item [($\textbf{c}_3$)]
   For the case that $m=1$ and $\mbox{rank}(B,AB,\cdots,A^{n-1}B)=n$,
   \cite[Theorem 15.5, pp.214]{Agrachev-Sachkov} obtained, by a  way  differing from ours,
    the following results: If $I$ is a closed interval in $(0,T^*_{x_0})$ with $|I|\leq \ln(1+\mbox{c}_A^{-1})$, then $\sharp[\mathbb{S}_{x_0}\cap I]\leq (n-1)$. (Here, $\mbox{c}_A\triangleq\max_{1\leq i\leq n}|c_i|$, with $\mbox{det}(\lambda\mathbb{I}_{n\times n} -A)=\lambda^n+c_1\lambda^{n-1}+\cdots+c_n$.)
    We do not know the connection between two numbers $\mbox{d}_{A}$ and $\ln(1+\mbox{c}_A^{-1})$ in general, but when $\sigma(A)\subset \mathbb{R}$, we clearly have $\ln(1+\mbox{c}_A^{-1})<+\infty=\mbox{d}_A$.

For the case that $m=(n-1)$, some related results are presents in \cite[Chapter 5]{Biolo}.

  \item [($\textbf{c}_4$)]  Let $k\triangleq \mbox{rank}(B,AB,\cdots,A^{n-1}B)$. Let
  $$
    \tilde{\mbox{q}}_{A,B}\triangleq \min\{j\in\mathbb{N}^+: \mbox{rank}(B,AB,\cdots,A^{j-1}B)=k\}.
$$
Then one can easily check that $\tilde{\mbox{q}}_{A,B}\leq \mbox{q}_{A,B}$. {\it A natural question is to ask
    if it is true that
    $$
    \sharp[\mathbb{S}_{x_0}\cap I]\leq \tilde{\mbox{q}}_{A,B}-1\;\;\mbox{for each open interval}\;\;I \;\;\mbox{with}\;\;|I|\leq d_A.
    $$ }
    We only know that  $\tilde{\mbox{q}}_{A,B}=\mbox{q}_{A,B}$ when $m=1$.

    \item [($\textbf{c}_5$)] The numbers  $\mbox{q}_{A,B}$ and $\mbox{d}_A$ were introduced in \cite{QSL2}, where
    the controllability of impulse controlled systems of heat equations coupled by constant matrices was studied.

\end{enumerate}

    The second main result concerns how the optimal control $u^*_{x_0}$ (to $(\mathcal{TP})_{x_0}$) moves
     as time $t$ varies.

\begin{theorem}\label{yu-theorem-9-19-1}
 Assume that $(\textbf{H})_{x_0}$ holds.
   Then the dynamic behaviour  of  the  optimal  control to $(\mathcal{TP})_{x_0}$
   has  one and only one of the following four  possibilities:
\begin{description}
  \item [($\textbf{A}_1$)] It stays in one direction on $\partial B_1(0)$ over  $(0,T^*_{x_0})$;

  \item [($\textbf{A}_2$)] It continuously moves on $\partial B_1(0)$, and in any subinterval of $(0,T^*_{x_0})$,
   never stays constantly in one direction;

  \item [($\textbf{A}_3$)] It is a non-constant-valued step  function  taking only two values  (on $\partial B_1(0)$),
  which have opposite directions;

  \item [($\textbf{A}_4$)] It is a  non-constant-valued
  piece-wise continuous function,  and in any subinterval of $(0,T^*_{x_0})$, never stays constantly in one direction.

\end{description}
       \end{theorem}

  About Theorem \ref{yu-theorem-9-19-1}, it deserves mentioning that in Subsection \ref{yu-section-5}, we show, by several examples, that
    all cases $(\textbf{A}_i)$, $i=1,\dots, 4$,  may happen.

\vskip 5pt
    The third main result concerns   $\mathbb{D}_{x_0}$.
   \begin{theorem}\label{theorem4.4,10-3}
    Suppose that $(\textbf{H})_{x_0}$ holds. Then the following statements are true:

\begin{description}
  \item[$(i)$] If $n\geq 2$ and $\mathbb{S}_{x_0}\neq \emptyset$, then
    \begin{equation}\label{yu-9-27-12}
    \mathbb{D}_{x_0}\subset \partial B_1(0)\bigcap\left[\bigcup_{j=1}^{\left[{n}/{2}\right]}
    (B^\top (A^\top)^{2j-1}\mathcal{H}_{2j-1})\right],
\end{equation}
where
$[s]\triangleq\max\{j\in\mathbb{N}:j\leq s\}$ and
\begin{equation}\label{yu-9-27-11}
\mathcal{H}_j\triangleq\mbox{span}\{B,AB,\ldots,A^{j}B\}\bigcap
    \mbox{span}\{B,AB,\ldots,A^{j-1}B\}^\bot,\;j=1,\dots,n-1.
\end{equation}
  \item[$(ii)$] If $n=1$, then $\mathbb{S}_{x_0}=\emptyset$.
\end{description}
    \end{theorem}
        Several notes on Theorem \ref{theorem4.4,10-3} are given as follows:
\begin{enumerate}

  \item[($\textbf{d}_1$)] By Theorem \ref{theorem4.4,10-3},
   in the case that $n\geq 2$, $\mbox{dim}[\mathcal{H}_{2j-1}]=0$ (for each $1\leq j\leq \left[\frac{n}{2}\right]$) and $(\textbf{H})_{x_0}$ holds, we have
  $\mathbb{S}_{x_0}=\emptyset$ (i.e., the optimal control is continuous over $(0,T^*_{x_0})$).
     In particular, when $\mbox{rank}(B)=n$, we have $\mathbb{S}_{x_0}=\emptyset$.

  \item[($\textbf{d}_2$)] We call a unit vector in $\mathcal{H}_j$ (with $1\leq j\leq (n-1)$) as  {\it a new direction of order $j$}, which is in the subspace:  $\mbox{span}\{B,AB,\ldots, A^jB\}$, and orthogonal to the subspace: $\mbox{span}\{B,AB,\ldots, A^{j-1}B\}$.
  From Theorem \ref{theorem4.4,10-3}, we see that $\mathbb{D}_{x_0}$ can only contain
  new directions of order $(2j-1)$ through the transformation $B^\top (A^\top)^{2j-1}$, with
    $1\leq j\leq \left[\frac{n}{2}\right]$.

  \item[($\textbf{d}_3$)] In Section \ref{appendix-9}, we further prove what follows
  (see Theorem \ref{yu-proposition-10-16-1} in Section \ref{appendix-9}):
  \begin{itemize}
   \item  In the case that $\mbox{rank}(B)=1$,
  $\mathbb{D}_{x_0}$ contains at most two vectors:
  $\frac{B^\top v}{\|B^\top v\|_{\mathbb{R}^m}}$ and $-\frac{B^\top v}{\|B^\top v\|_{\mathbb{R}^m}}$,
  where $v$ is a unit vector in  the $1$-dim space: $\mbox{ker}(B^\top)^\bot$.
  \item In the case that  $\mbox{rank}(B,AB,\ldots,A^{n-1}B)=n$ and $\mbox{rank}(B)=n-1$,  $\mathbb{D}_{x_0}$ contains at most two vectors:
  $\frac{B^\top A^\top v}{\|B^\top A^\top v\|_{\mathbb{R}^m}}$ and $-\frac{B^\top A^\top v}{\|B^\top A^\top v\|_{\mathbb{R}^m}}$, where $v$ is a unit vector in the $1$-dim space: $\mbox{ker}(B^\top)$.
  \end{itemize}

\end{enumerate}

\noindent{\bf Plan of the paper.} Section \ref{section2,11-3} presents proofs of the main theorems;
Section \ref{example} gives some examples;  Section \ref{appendix-9}
studies further information on  switching directions; Section \ref{appendix,11-3} (Appendix) provides some
results used in our studies.

\section{Proofs of main results}\label{section2,11-3}

\subsection{Proof of Theorem \ref{main-theorem}}
\par
  The next Lemma \ref{main-theorem-lemma1} plays an important role in the study of the upper bound for the number of switching points.
  The key to prove Lemma \ref{main-theorem-lemma1} is
    \cite[Theorem 2.2]{QSL2}.
\begin{lemma}\label{main-theorem-lemma1}
Let $\mbox{d}_A$ and $\mbox{q}_{A,B}$ be given by (\ref{A-spectrum-distance}) and (\ref{yu-9-13-b-2}) respectively. Then for each
\begin{equation*}\label{main-theorem-lemma1-2}
    z\in \mbox{span}\{B,AB,\ldots,A^{n-1}B\}\setminus\{0\},
\end{equation*}
     the following nodal set:
\begin{eqnarray}\label{main-theorem-lemma1-1}
 \mathcal O_z  \triangleq \{t\in\mathbb{R}^+ :  B^\top e^{A^\top t} z=0 \}
\end{eqnarray}
has at most $(\mbox{q}_{A,B}-1)$ elements over each open interval $I$ with  $|I|\leq d_A$.

\end{lemma}

\begin{proof}
Arbitrarily fix a vector $z\in\mbox{span}\{B,AB,\ldots,A^{n-1}B\}\setminus\{0\}$
and  an open interval $I\subset \mathbb R$
with  $|I|\leq \mbox{d}_A$. We only need to prove
\begin{equation}\label{main-theorem-lemma1-3}
  \sharp\;[\mathcal O_z\cap I] \leq (\mbox{q}_{A,B}-1).
\end{equation}
Suppose by contradiction that
$\sharp\;[\mathcal O_z\cap I] \geq \mbox{q}_{A,B}$.
Then by  (\ref{main-theorem-lemma1-1}), there is an strictly increasing sequence $\{t_j\}_{j=1}^{\mbox{q}_{A,B}}\subset I$ so that
$B^\top e^{A^\top t_j}z=0$ for each $j\in\{1,2,\ldots,\mbox{q}_{A,B}\}$.
This yields that for each $\{u_j\}_{j=1}^{\mbox{q}_{A,B}}\subset \mathbb R^m$,
\begin{equation*}
\left\langle\sum_{j=1}^{\mbox{q}_{A,B}}  e^{A t_j}Bu_j, z\right\rangle_{\mathbb R^m}
=\sum_{j=1}^{\mbox{q}_{A,B}} \langle u_j, B^\top e^{A^\top t_j} z\rangle_{\mathbb R^m}=0,
\end{equation*}
which leads to
\begin{equation}\label{yu-10-9-1}
z\notin \mbox{span} \left\{e^{At_1}B,\ldots,e^{At_{\mbox{q}_{A,B}}}B\right\}.
\end{equation}
    Meanwhile, by the Hamilton-Cayley theorem, we have
\begin{equation*}
    \mbox{span} \left\{e^{At_1}B,\ldots,e^{At_{\mbox{q}_{A,B}}}B\right\}\subset
    \mbox{span} \left\{B,AB,\ldots,A^{n-1}B\right\}.
\end{equation*}
    This, along with (\ref{yu-10-9-1}), yields
\begin{equation}\label{yu-10-9-2}
    \mbox{span} \left\{e^{At_1}B,\ldots,e^{At_{\mbox{q}_{A,B}}}B\right\}\Subset
    \mbox{span} \left\{B,AB,\ldots,A^{n-1}B\right\}.
\end{equation}
    On the other hand, since $t_{\mbox{q}_{A,B}}-t_1<\mbox{d}_A$, we can apply \cite[Theorem 2.2]{QSL2}
    to see
\begin{equation*}
    \mbox{span} \left\{e^{At_1}B,\ldots,e^{At_{\mbox{q}_{A,B}}}B\right\}=
    \mbox{span} \left\{B,AB,\ldots,A^{n-1}B\right\},
\end{equation*}
    which contradicts (\ref{yu-10-9-2}). So (\ref{main-theorem-lemma1-3}) is true. This ends the proof.
\end{proof}
 The  next Proposition \ref{pontryagin-maximum-principle} presents some properties of
   $(\mathcal{TP})_{x_0}$. They  seem to be well-known. However, we do not find
   a suitable reference (in particular, $(iii)$ with \eqref{2.1-9-29} in Proposition \ref{pontryagin-maximum-principle}). For the sake of the completeness of the paper, we will give its proof in Appendix
   (see Subsection \ref{appendix-2}).

\begin{proposition}\label{pontryagin-maximum-principle}
 Suppose that $(\textbf{H})_{x_0}$ holds. Then the following conclusions are true:
\begin{description}
  \item[$(i)$] It stands that $0<T^*_{x_0}<+\infty$.
  \item[$(ii)$] The problem $(\mathcal{TP})_{x_0}$ has an optimal control.
  \item[$(iii)$] There exists $z^*\in \mbox{span}\{B,AB,\ldots,A^{n-1}B\}$ with
  \begin{eqnarray}\label{2.1-9-29}
B^\top e^{A^\top \left(T^*_{x_0}-\cdot\right) } z^* \neq 0
\;\;\mbox{in}\;\;
\mathcal{C}((0,T^*_{x_0});\mathbb R^m),
\end{eqnarray}
   so that  any optimal control $u^*_{x_0}$ (to $(\mathcal{TP})_{x_0}$)  satisfies
\begin{eqnarray}\label{pontryagin-maximum-principle-1}
 \left\langle u^*_{x_0}(t),B^\top e^{ A^\top \left(T^*_{x_0}-t\right) } z^* \right\rangle_{\mathbb R^m}
 =\max_{v\in B_1(0)}   \left\langle v,B^\top e^{ A^\top \left(T^*_{x_0}-t\right) } z^* \right\rangle_{\mathbb R^m}\;
 \;\mbox{a.e.}\; t\in (0,T^*_{x_0}).
\end{eqnarray}
\item [$(iv)$] If  $u^*_{x_0}$ is an optimal control to $(\mathcal{TP})_{x_0}$, then, in $(0,T^*_{x_0})$, 
\begin{eqnarray}\label{main-theorem-lemma2-1}
u^*_{x_0}(t)=
\begin{cases}
    \frac{B^\top e^{A^\top(T^*_{x_0}- t)} z^*}{\|B^\top e^{A^\top(T^*_{x_0}- t) } z^* \|_{\mathbb R^m}},&t\in (0,T^*_{x_0})\setminus(T^*_{x_0}-\mathcal{O}_{z^*}),\\
    \lim_{s\to t^-}\frac{B^\top e^{A^\top(T^*_{x_0}- s)} z^*}{\|B^\top e^{A^\top(T^*_{x_0}- s) } z^* \|_{\mathbb R^m}},&t\in (T^*_{x_0}-\mathcal{O}_{z^*})\bigcap(0,T^*_{x_0}),
\end{cases}
\end{eqnarray}
    where $z^*$ is given in (iii) and $\mathcal{O}_{z^*}$ is given by (\ref{main-theorem-lemma1-1}) with $z=z^*$.

\item [$(v)$] Any optimal control $u^*_{x_0}$ (to $(\mathcal{TP})_{x_0}$) has the bang-bang property: $\|u^*_{x_0}(t)\|_{\mathbb{R}^m}=1$ for each $t\in (0,T^*_{x_0})$.

    \item [$(vi)$] The optimal control to $(\mathcal{TP})_{x_0}$ is unique.
\end{description}
\end{proposition}

\begin{remark}
      From  (\ref{main-theorem-lemma2-1}), we see that under  $(\textbf{H})_{x_0}$,
     $\mathbb{S}_{x_0}\subset (T^*_{x_0}-\mathcal{O}_{z^*})\cap(0,T^*_{x_0})$.
     It deserves mentioning that $\mathbb{S}_{x_0}\neq(T^*_{x_0}-\mathcal{O}_{z^*})\cap(0,T^*_{x_0})$ in general.
      This can be seen from the example given in Subsection \ref{subsection-11-8-1}.
\end{remark}

With the aid of Proposition \ref{pontryagin-maximum-principle} and Lemma \ref{main-theorem-lemma1}, we are ready to prove
Theorem \ref{main-theorem}.

\begin{proof}[Proof of Theorem \ref{main-theorem}]
Let $z^*$ be given by $(iii)$ of Proposition \ref{pontryagin-maximum-principle}, $\mathcal O_{z^*}$ be given by (\ref{main-theorem-lemma1-1}) with $z=z^*$. By Proposition \ref{pontryagin-maximum-principle}, let $u^*_{x_0}$  be the optimal control to $(\mathcal{TP})_{x_0}$.

 We first  prove  $(i)$ in Theorem \ref{main-theorem}. Let $I\subset(0,T^*_{x_0})$ be an open interval, with
 $|I|\leq\mbox{d}_A$. Write
 \begin{eqnarray*}
 \mathbb{S}_{x_0}(I)\triangleq  \mathbb{S}_{x_0}\cap I.
 \end{eqnarray*}
 Then it follows by  \eqref{1.7-9-28-b}, \eqref{main-theorem-lemma1-1}
 and \eqref{main-theorem-lemma2-1} that
\begin{equation}\label{main-theorem-2}
  \mathbb{S}_{x_0}(I) \subset (T^*_{x_0}-\mathcal O_{z^*})\cap I.
\end{equation}
(\emph{Notice that \eqref{main-theorem-2} holds for any open interval $I$.})
Meanwhile,
because $z^*\in \mbox{span} \{B,AB,\ldots,A^{n-1}B\}\setminus\{0\}$  and
$|T^*_{x_0}-I|\leq \mbox{d}_A$, we can apply Lemma \ref{main-theorem-lemma1}
(with $z=z^*$) to find that
\begin{equation*}
  \sharp[(T^*_{x_0}-\mathcal O_{z^*}) \cap I]=\sharp[\mathcal O_{z^*}\cap (T^*_{x_0}-I)] \leq (\mbox{q}_{A,B}-1).
  \end{equation*}
This, along with (\ref{main-theorem-2}), yields
$\sharp\mathbb{S}_{x_0}(I) \leq \mbox{q}_{A,B}-1$.
Hence,  $u^*_{x_0}$ has at most $(\mbox{q}_{A,B}-1)$ switching points over $I$, which leads to the conclusion $(i)$ in Theorem \ref{main-theorem}.

We next show the conclusion $(ii)$ in Theorem \ref{main-theorem}.
Arbitrarily fix $\hat t\in \mathbb{S}_{x_0}$. By \eqref{main-theorem-lemma2-1},
we see that $\hat t\in (T^*_{x_0}-\mathcal{O}_{z^*})\cap (0,T^*_{x_0})$, which contains at most finitely many elements (see Lemma \ref{main-theorem-lemma1}). Thus,  it follows from $(iv)$ in Proposition  \ref{pontryagin-maximum-principle}
that there is $\hat \varepsilon \in (0,\hat t)$ so that
\begin{equation}\label{main-theorem-3}
  u^*_{x_0}(t)=
 \frac{ B^\top e^{ A^\top (T^*_{x_0}- t) } z^*}{\| B^\top e^{ A^\top (T^*_{x_0}- t) } z^*\|_{\mathbb R^m}}
  \;\;\mbox{for each}\;\; t\in(\hat t-\hat\varepsilon, \hat t+\hat\varepsilon)\setminus \{\hat{t}\}.
\end{equation}

 Because the function: $\mathbb R\ni t\mapsto B^\top e^{A^\top (T^*_{x_0}- t)} z^*\in\mathbb{R}^n$, is real analytic
and is not identically zero (see (\ref{2.1-9-29})),  there exists $j\in\mathbb N$ and $a_j\in\mathbb R^m\setminus\{0\}$ so that when $t$ is near $\hat t$,
\begin{equation*}
B^\top e^{ A^\top (T^*_{x_0}- t)} z^* = a_j(t-\hat t)^j+b_j(t)(t-\hat t)^{j+1}
= (t-\hat t)^j\left[a_j +b_j(t)(t-\hat t)\right],
\end{equation*}
    where $b_j(\cdot):\mathbb{R}^+\to\mathbb{R}^m$ is bounded and real analytical in $(\hat t-\varepsilon, \hat t+\varepsilon)$
    for each $\varepsilon\in (0,\hat \varepsilon]$.
This, together with (\ref{main-theorem-3}), leads to  the following two conclusions:

\begin{description}
  \item[(a)] When $j$ is even, $\lim_{t\rightarrow \hat t^+}u^*_{x_0}(t)=\lim_{t\rightarrow \hat t^-}u^*_{x_0}(t)
             =\frac{a_j}{\|a_j\|_{\mathbb R^m}}$;
   \item[(b)] When $j$ is odd,  $\lim_{t\rightarrow \hat t^+}u^*_{x_0}(t) = -\lim_{t\rightarrow \hat t^-}u^*_{x_0}(t)
               =\frac{a_j}{\|a_j\|_{\mathbb R^m}}$.
 \end{description}

Finally,  since $\hat t$ is a switching point of $u^*_{x_0}$, it follows by Definition \ref{Switching point} that
\begin{equation*}
  \lim_{t\rightarrow \hat t^+}u^*_{x_0}(t)\neq\lim_{t\rightarrow \hat t^-}u^*_{x_0}(t),
\end{equation*}
which, along with the above conclusions $(a)$ and $(b)$, leads to that $(a)$ is impossible since $\hat{t}$ is a switching point. Thus, by $(b)$, we obtain \eqref{main-theorem-1}. This ends the proof of
Theorem \ref{main-theorem}.
\end{proof}

\subsection{Proof of Theorem \ref{yu-theorem-9-19-1}}

The key to prove Theorem \ref{yu-theorem-9-19-1} is the following theorem:
\begin{theorem}\label{yu-lemma-6-9-1}
Suppose that $(\textbf{H})_{x_0}$ holds.
Let $u^*_{x_0}$ be the  optimal control to $(\mathcal{TP})_{x_0}$. If there is a vector $v\in\partial B_1(0)$
and a subinterval $(a,b)\subset (0,T^*_{x_0})$ (with $a<b$) so that $u^*_{x_0}(t)=v$ for all $t\in (a,b)$, then
when $t\in (0,T^*_{x_0})$, $u^*_{x_0}(t)$ is  either $v$ or $-v$.
\end{theorem}
\begin{proof}
    Let $z^*\in\mbox{span}\{B,AB,\ldots,A^{n-1}B\}\setminus\{0\}$ be given by $(iii)$ in Proposition \ref{pontryagin-maximum-principle}.
    By Lemma  \ref{main-theorem-lemma1}, $\mathcal{O}_{z^*}\cap(0,T^*_{x_0})$ contains at most finitely many elements.
    So we can assume, without loss of generality, that
    \begin{eqnarray}\label{2.11,10-19}
    (a,b)\bigcap (T^*_{x_0}-\mathcal{O}_{z^*})=\emptyset.
    \end{eqnarray}
    (For otherwise, we can take $(a_1,b_1)\subset (a,b)$ (with $a_1<b_2$) so that $ (a_1,b_1)\cap (T^*_{x_0}-\mathcal{O}_{z^*})=\emptyset$.)
    Several facts are given in order. First, by the assumption of Theorem \ref{yu-lemma-6-9-1}, we have
    \begin{equation}\label{yu-6-22-2}
    u^*_{x_0}(t)=v\;\;\mbox{for each}\;t\in(a,b).
\end{equation}
    Second, it follows by (\ref{main-theorem-lemma2-1}) and \eqref{2.11,10-19} that
    \begin{eqnarray}\label{2.13,10-19}
     u^*_{x_0}(t)=\frac{B^\top e^{A^{\top}(T^*_{x_0}-t)}z^*}{\|B^\top e^{A^{\top}(T^*_{x_0}-t)}z^*\|_{\mathbb{R}^m}}
     \;\;\mbox{for each}\;t\in(a,b).
    \end{eqnarray}
    Third, from \eqref{yu-6-22-2} and \eqref{2.13,10-19}, we find
\begin{eqnarray*}
        B^\top e^{A^{\top}(T^*_{x_0}-t)}z^*
        =
\|B^\top e^{A^{\top}(T^*_{x_0}-t)}z^*\|_{\mathbb{R}^m}v\;\;\mbox{for each}\;t\in(a,b).
\end{eqnarray*}
       Multiplying the above by  $v$ and noting that $v\in \partial B_1(0)$, we see
\begin{equation}\label{yu-6-22-3}
    \langle v,B^\top e^{A^{\top}(T^*_{x_0}-t)}z^*\rangle_{\mathbb{R}^m}
    =\|B^\top e^{A^{\top}(T^*_{x_0}-t)}z^*\|_{\mathbb{R}^m}
    \;\;\mbox{for any}\;t\in(a,b).
\end{equation}

    We now define a function $f(\cdot): \mathbb{R}\rightarrow\mathbb{R}^m$ by
\begin{equation}\label{yu-6-22-4}
    f(t)\triangleq B^\top e^{A^{\top}(T^*_{x_0}-t)}z^*-\langle v,B^\top e^{A^{\top}(T^*_{x_0}-t)}z^*\rangle_{\mathbb{R}^m} v,
    \; t\in\mathbb{R}.
\end{equation}
By \eqref{yu-6-22-4}, \eqref{yu-6-22-3}, \eqref{yu-6-22-2} and \eqref{2.13,10-19}, we obtain that for each $t\in (a,b)$,
\begin{eqnarray}\label{2.16,10-19}
f(t)
&=& B^\top e^{A^{\top}(T^*_{x_0}-t)}z^*-\|B^\top e^{A^{\top}(T^*_{x_0}-t)}z^*\|_{\mathbb{R}^m}u^*_{x_0}(t)\nonumber\\
&=& B^\top e^{A^{\top}(T^*_{x_0}-t)}z^*-B^\top e^{A^{\top}(T^*_{x_0}-t)}z^*=0.
\end{eqnarray}
Meanwhile,  from \eqref{yu-6-22-4}, we see that $f(\cdot)$ is real analytic over $\mathbb{R}$. This, along with
\eqref{2.16,10-19} and \eqref{yu-6-22-4}, yields
\begin{equation}\label{yu-6-22-6}
   B^\top e^{A^{\top}(T^*_{x_0}-t)}z^*=\langle v,B^\top e^{A^{\top}(T^*_{x_0}-t)}z^*\rangle_{\mathbb{R}^m} v
    \;\;\mbox{for any}\;t\in\mathbb{R}.
\end{equation}
   Multiplying the above by $B^\top e^{A^{\top}(T^*_{x_0}-t)}z^*$, we get
   \begin{eqnarray*}
   \|B^\top e^{A^{\top}(T^*_{x_0}-t)}z^*\|_{\mathbb{R}^m}^2=\langle v, B^\top e^{A^{\top}(T^*_{x_0}-t)}z^*\rangle_{\mathbb{R}^m}^2
    \;\;\mbox{for each}\;t\in(0,T^*_{x_0}),
   \end{eqnarray*}
    which yields that for each $t\in(0,T^*_{x_0})$
    \begin{equation}\label{yu-6-22-7}
    \langle v, B^\top e^{A^{\top}(T^*_{x_0}-t)}z^*\rangle_{\mathbb{R}^m}
    =\pm
    \|B^\top e^{A^{\top}(T^*_{x_0}-t)}z^*\|_{\mathbb{R}^m}.
       \end{equation}

    Finally, we show that for arbitrarily fixed $t\in (0,T^*_{x_0})$,
    \begin{eqnarray}\label{wang2.19}
    u^*_{x_0}(t)=\pm v.
    \end{eqnarray}
    Indeed, if  $t\in (0,T^*_{x_0})\setminus(T^*_{x_0}-\mathcal{O}_{z^*})$, then we get from
    (\ref{main-theorem-lemma2-1}), \eqref{yu-6-22-6} and \eqref{yu-6-22-7} that
    \begin{eqnarray}\label{2.19,10-19}
    u^*_{x_0}(t)=\pm \frac{\|B^\top e^{A^{\top}(T^*_{x_0}-t)}z^*\|_{\mathbb{R}^m} v}{\|B^\top e^{A^{\top}(T^*_{x_0}-t)}z^*\|_{\mathbb{R}^m}},
    \end{eqnarray}
    which leads to \eqref{wang2.19} for the case that $t\in(0,T^*_{x_0})\setminus(T^*_{x_0}-\mathcal{O}_{z^*})$.
        We next consider the case that $t\in(T^*_{x_0}-\mathcal{O}_{z^*})\cap(0,T^*_{x_0})$. By (\ref{main-theorem-lemma2-1}), we have
    \begin{eqnarray}\label{2.20,10-19}
    u^*_{x_0}(t)=\lim_{s\to t^-}\frac{B^\top e^{ A^\top(T^*_{x_0}- s)} z^*}{\|B^\top e^{ A^\top(T^*_{x_0}- s) } z^* \|_{\mathbb R^m}}.
    \end{eqnarray}
    Two observations are given in order. First, since $(T^*_{x_0}-\mathcal{O}_{z^*})\cap (0,T^*_{x_0})$ contains at most finitely many elements, for each $t\in (T^*_{x_0}-\mathcal{O}_{z^*})\cap (0,T^*_{x_0})$, we can find $\delta\in (0,t)$ so that $(\delta,t)\subset  (0,T^*_{x_0})\setminus(T^*_{x_0}-\mathcal{O}_{z^*})$. Consequently,
        $u^*_{x_0}$ is continuous over $(\delta, t)$.
        Second,  we can use the same way in the proof of \eqref{2.19,10-19} to verify that  for each $s\in (\delta,t)$, $u^*_{x_0}(s)=\pm v$.
       From these observations, we find
    either $u^*_{x_0}(s)=v$ for all $s\in (\delta,t)$ or $u^*_{x_0}(s)=-v$ for all $s\in (\delta,t)$.
    This, along with
     \eqref{2.20,10-19}, leads to \eqref{wang2.19} for the case when $t\in(T^*_{x_0}-\mathcal{O}_{z^*})\cap(0,T^*_{x_0})$.

    Thus, we complete the proof of Theorem \ref{yu-lemma-6-9-1}.
\end{proof}
\begin{proof}[The proof of Theorem \ref{yu-theorem-9-19-1}]
    Let $u^*_{x_0}$ be the optimal control to $(\mathcal{TP})_{x_0}$.
    First of all, by $(iv)$ of Proposition  \ref{pontryagin-maximum-principle}, we have
    $\|u^*_{x_0}(t)\|_{\mathbb{R}^m}=1$ for all $t\in(0,T^*_{x_0})$, i.e., $u^*_{x_0}(t)\in\partial B_1(0)$ for all $t\in(0,T^*_{x_0})$. The rest of the proof is divided into four steps.
\vskip 5pt
  \noindent {\it Step 1. We show that if $\mathbb{S}_{x_0}=\emptyset$ and $u^*_{x_0}$ is constant-valued in $(0,T^*_{x_0})$, then $(\textbf{A}_1)$ is true.}

 Since $u^*_{x_0}(t)\in \partial B_1(0)$ for all $t\in (0,T^*_{x_0})$, the conclusion is obvious.

\vskip 5pt
  \noindent{\it Step 2. We show that if $\mathbb{S}_{x_0}=\emptyset$ and $u^*_{x_0}$ is not constant-valued in $(0,T^*_{x_0})$, then $(\textbf{A}_2)$ is true.}

    By contradiction, we suppose that $(\textbf{A}_2)$ is not true in this case. Then there is a  subinterval
    $(a,b)\subset(0,T^*_{x_0})$ and $v\in\partial B_1(0)$ so that
\begin{equation}\label{yu-10-19-1-b}
    u^*_{x_0}(t)=v\;\;\mbox{for all}\;t\in(a,b).
\end{equation}
   From \eqref{yu-10-19-1-b}, we can use  Theorem \ref{yu-lemma-6-9-1} to find that
    for each $t\in(0,T^*_{x_0})$,
   \begin{eqnarray}\label{2.23,10-21}
   u^*_{x_0}(t)=\pm v.
   \end{eqnarray}
   Meanwhile, since $\mathbb{S}_{x_0}=\emptyset$, we have that $u^*_{x_0}$ is continuous over $(0,T^*_{x_0})$.
   This, along with \eqref{2.23,10-21} and \eqref{yu-10-19-1-b}, shows that $u^*_{x_0}(t)=v$
   for all $t\in (0,T^*_{x_0})$, which leads to a contradiction. Hence,  $(\textbf{A}_2)$ is true.
\vskip 5pt
    \noindent {\it Step 3. We show that if $\mathbb{S}_{x_0}\neq \emptyset$ and $u^*_{x_0}$ is a step  function  taking only two values  (on $\partial B_1(0)$) which have
   opposite directions, then $(\textbf{A}_3)$ is true.}

   In this case, $(\textbf{A}_3)$ follows from  Theorem \ref{main-theorem} at once.

\vskip 5pt
 \noindent{\it Step 4. We prove that if $\mathbb{S}_{x_0}\neq \emptyset$ and $u^*_{x_0}$ is not a step  function,  then $(\textbf{A}_4)$ is true.}
\par
   By contradiction, we suppose that $(\textbf{A}_4)$ is not true in this case. Then
    there is a subinterval
    $(a,b)\subset (0,T^*_{x_0})$ and $v\in\partial B_1(0)$ such that  (\ref{yu-10-19-1-b}) holds.
    Then by Theorem \ref{yu-lemma-6-9-1}, we have \eqref{2.23,10-21}.
    Since $\mathbb{S}_{x_0}\neq \emptyset$ and $u^*_{x_0}$
   is  piece-wise continuous over $(0,T^*_{x_0})$, we see from \eqref{2.23,10-21} that $u^*_{x_0}$
   is  a step function, which leads to a contradiction. So $(\textbf{A}_4)$ is true.

\par
    Thus, we end the proof of Theorem \ref{yu-theorem-9-19-1}.

\subsection{The proof of Theorem \ref{theorem4.4,10-3}}
   We first show the conclusion $(i)$.
   Because of $(\textbf{H})_{x_0}$, all conclusions in Proposition \ref{pontryagin-maximum-principle} hold. Let $u^*_{x_0}$ be the  optimal control to $(\mathcal{TP})_{x_0}$.
Let  $z^*\in\mathbb{R}^n\setminus\{0\}$ be given by  $(iii)$ of Proposition \ref{pontryagin-maximum-principle}.
    Set
\begin{eqnarray}\label{new2.24,10-22}
 \hat{z}^*\triangleq e^{A^\top T^*_{x_0}}z^*.
 \end{eqnarray}
       Arbitrarily fix $\hat{t}\in\mathbb{S}_{x_0}$. (Notice that we assumed $\mathbb{S}_{x_0}\neq \emptyset$.)
    and
     then define
    \begin{equation}\label{yu-10-19-3-b}
    k(\hat{t};\hat{z}^*)\triangleq \min\{j\in\mathbb{N}:B^\top (A^\top)^je^{-A^\top \hat{t}}\hat{z}^*\neq 0\}.
\end{equation}
\par
The rest of the proof is organized by two steps.
 \vskip 5pt
   \noindent {\it Step 1. We prove that
   \begin{equation}\label{yu-10-19-3}
    k(\hat{t};\hat{z}^*)\leq (n-1)
\end{equation}
    and
\begin{equation}\label{yu-10-27-3}
    k(\hat{t};\hat{z}^*)\;\;\mbox{is an odd number}.
\end{equation}}
\par
    We start with proving (\ref{yu-10-19-3}).
       By contradiction, we suppose that  $k(\hat{t};\hat{z}^*)>(n-1)$.
       Then by \eqref{yu-10-19-3-b}, we have that
      $B^\top (A^\top)^je^{-A^\top \hat{t}}\hat{z}^*=0$ for $j=0,1,\ldots,n-1$. This, along with the Hamilton-Cayley theorem, yields that
   \begin{eqnarray}\label{2.36,10-25}
   B^\top (A^\top)^je^{-A^\top \hat{t}}\hat{z}^*=0\;\;\mbox{for each}\;\;j\in\mathbb{N}.
   \end{eqnarray}
  Since
\begin{equation}\label{yu-10-23-1}
    B^\top e^{-A^\top t}\hat{z}^*=\sum_{j=0}^\infty
    \frac{B^\top (-A^\top)^je^{-A^\top \hat{t}}\hat{z}^*}{j!}(t-\hat{t})^j\;\;\mbox{for any}\;\;t\in\mathbb{R}^+,
\end{equation}
  it follows from  \eqref{2.36,10-25} and (\ref{new2.24,10-22}) that
  \begin{eqnarray*}
  B^\top e^{A^\top(T^*_{x_0}-t)}z^*=B^\top e^{-A^\top t}\hat{z}^*=0\;\;\mbox{for all}\;\;t\in(0,T^*_{x_0}),
   \end{eqnarray*}
  which contradicts (\ref{2.1-9-29}) (in Proposition \ref{pontryagin-maximum-principle}). Thus (\ref{yu-10-19-3}) holds.

\par
    We next show  (\ref{yu-10-27-3}). Two observations are given in order. First,
  it follows by
(\ref{yu-10-23-1}) and  (\ref{yu-10-19-3-b}) that
\begin{eqnarray}\label{yu-10-19-4}
     B^\top e^{-A^\top t}\hat{z}^*
    =(t-\hat{t})^{k(\hat{t};\hat{z}^*)}\Big[a_{k(\hat{t};\hat{z}^*)}
     +b_{k(\hat{t};\hat{z}^*)}(t)(t-\hat{t})\Big]\;\; \mbox{for each}\;\; t\in \mathbb{R}^+,
\end{eqnarray}
    where
$$
    a_{k(\hat{t};\hat{z}^*)}\triangleq\frac{B^\top(-A^\top)^{k(\hat{t};\hat{z}^*)}
    e^{-A^\top\hat{t}}\hat{z}^*}{k(\hat{t};\hat{z}^*)!},
    \;\;\;\;b_{k(\hat{t};\hat{z}^*)}(t)\triangleq\sum^{+\infty}_{j= k(\hat{t};\hat{z}^*)+1}\frac{B^\top(-A^\top)^je^{-A^\top\hat{t}}\hat{z}^*}{j!}
    (t-\hat{t})^{j-k(\hat{t};\hat{z}^*)-1}.
$$
  Second, since $\hat{t}\in\mathbb{S}_{x_0}$, we see that  (\ref{main-theorem-3}) holds for some $\hat \varepsilon>0$.

  By contradiction, we suppose that $k(\hat{t};\hat{z}^*)$ is even. Then
  it follows by  (\ref{main-theorem-3}), (\ref{new2.24,10-22}) and (\ref{yu-10-19-4}) that
\begin{equation*}
    \lim_{t\to \hat{t}^+}u^*_{x_0}(t)=\lim_{t\to \hat{t}^-}u^*_{x_0}(t)
    =\frac{a_{k(\hat{t};\hat{z}^*)}}{\|a_{k(\hat{t};\hat{z}^*)}\|_{\mathbb{R}^m}},
\end{equation*}
   which contradicts to  the fact that $\hat{t}\in\mathbb{S}_{x_0}$. Therefore, $k(\hat{t};\hat{z}^*)$ is odd, i.e., (\ref{yu-10-27-3}) holds.
\vskip 5pt

\noindent \emph{Step 2. We  complete the proof.}

Write
\begin{equation}\label{yu-10-27-5}
    e^{-A^\top\hat{t}}\hat{z}^*=f_1+f_2,
\end{equation}
where
\begin{equation}\label{yu-10-24-1}
    f_1\in \mbox{span}\{B,AB,\ldots,A^{k(\hat{t};z)}B\}\;\;\mbox{and}\;\;f_2\in
    \mbox{span}\{B,AB,\ldots,A^{k(\hat{t};z)}B\}^\bot.
\end{equation}
\par
    We claim that
\begin{equation}\label{yu-10-19-10}
    f_1\in \mathcal{H}_{k(\hat{t};\hat{z}^*)}\setminus\{0\}\;\;\mbox{and}\;\;B^\top(A^\top)^{k(\hat{t};\hat{z}^*)}f_1\neq 0,
\end{equation}
    and
\begin{equation}\label{yu-10-27-4}
    \lim_{t\to \hat t^-}u^*_{x_0}(t)=\frac{B^\top(A^\top)^{k(\hat{t};\hat{z}^*)}f_1}
    {\|B^\top(A^\top)^{k(\hat{t};\hat{z}^*)}f_1\|_{\mathbb{R}^m}}.
\end{equation}
       When they are proved, we get from (\ref{yu-10-19-10})$_1$  (Here and in what follows, $(*)_j$ denotes the $j$-th equation  in
$(*)$.) and (\ref{yu-10-27-4}) that
$$
    \lim_{t\to \hat t^-}u^*_{x_0}(t)\in\partial B_1(0)\bigcap [B^\top (A^\top)^{k(\hat{t};\hat{z}^*)}\mathcal{H}_{k(\hat{t};\hat{z}^*)}].
$$
    This, along with (\ref{yu-10-27-3}) and (\ref{yu-10-19-3}), yields that
\begin{equation}\label{yu-10-27-1-b}
    \lim_{t\to \hat t^-}u^*_{x_0}(t)\in\partial B_1(0)\bigcap\left[\bigcup_{j=1}^{\left[{n}/{2}\right]}
    (B^\top (A^\top)^{2j-1}\mathcal{H}_{2j-1})\right].
\end{equation}
   Since $\hat t$ was arbitrarily taken from $\mathbb{S}_{x_0}$, we get (\ref{yu-9-27-12}) from
        (\ref{yu-9-27-10}) and (\ref{yu-10-27-1-b}) at once.
\par
    We now  prove (\ref{yu-10-19-10}). By  (\ref{yu-10-19-3-b}), we have
\begin{equation}\label{yu-10-19-5}
    B^\top (A^\top)^{k(\hat{t};\hat{z}^*)}e^{-A^\top \hat{t}}\hat{z}^*\neq 0\;\mbox{and}\;B^\top (A^\top)^je^{-A^\top \hat{t}}\hat{z}^*=0\;\;\mbox{for all}\;j=0,1,\ldots,k(\hat{t};\hat{z}^*)-1.
\end{equation}
   Meanwhile, we have
\begin{equation}\label{yu-9-28-5}
    \bigcap_{j=0}^{p}\mbox{ker}(B^\top(A^\top)^j)
    =\mbox{span}\{B,AB,\ldots,A^pB\}^\bot\;\;\mbox{for each}\;p=0,1,\ldots,n-1.
\end{equation}
    Then by (\ref{yu-10-27-5}), (\ref{yu-10-24-1})$_2$ and (\ref{yu-9-28-5}), we see
\begin{equation}\label{yu-10-19-7}
    B^\top (A^\top)^{j}e^{-A^\top\hat{t}}\hat{z}^*=B^\top (A^\top)^{j}f_1\;\;\mbox{for all}\;j=0,1,\ldots,k(\hat{t};\hat{z}^*).
\end{equation}
    This, along with (\ref{yu-10-19-5})$_2$, yields that
    $f_1\in \left[\bigcap_{j=0}^{k(\hat{t};\hat{z}^*)-1}\mbox{ker}(B^\top(A^\top)^j)\right]\setminus\{0\}$.
    From this and  (\ref{yu-9-28-5}), we obtain
\begin{equation*}\label{yu-10-27-6}
    f_1\in \mbox{span}\{B,AB,\ldots,A^{k(\hat{t};\hat{z}^*)-1}B\}^\bot\setminus\{0\},
\end{equation*}
   which, along with (\ref{yu-10-19-5}), (\ref{yu-10-24-1})$_1$ and  (\ref{yu-9-27-11}), leads to (\ref{yu-10-19-10})$_1$.
    By (\ref{yu-10-19-5})$_1$ and (\ref{yu-10-19-7}), we obtain (\ref{yu-10-19-10})$_2$ at once.

\par
    Next, we show (\ref{yu-10-27-4}). Since $\hat{t}\in\mathbb{S}_{x_0}$, it follows from \eqref{main-theorem-2}
    that $\hat t\in  (T^*_{x_0}-\mathcal{O}_{z^*})\cap (0,T^*_{x_0})$. From this, (\ref{main-theorem-lemma2-1}), (\ref{yu-10-19-4}) and
    (\ref{yu-10-27-3}), we get
\begin{equation*}\label{yu-10-27-9}
    \lim_{t\to \hat t^-}u^*_{x_0}(t)=
    \lim_{t\to \hat t^-}\frac{B^\top e^{-A^\top t}\hat{z}^*}
    {\|B^\top e^{-A^\top t}\hat{z}^*\|_{\mathbb{R}^m}}=\frac{B^\top(A^\top)^{k(\hat{t};\hat{z}^*)}e^{-A^\top\hat{t}}\hat{z}^*}
    {\|B^\top(A^\top)^{k(\hat{t};\hat{z}^*)}e^{-A^\top\hat{t}}\hat{z}^*\|_{\mathbb{R}^m}}.
\end{equation*}
    This, along with (\ref{yu-10-19-7}), yields (\ref{yu-10-27-4}). So we have proved the conclusion $(i)$ of Theorem \ref{theorem4.4,10-3}.
    \vskip 5pt

   We now show the conclusion $(ii)$.  In the case that $n=1$,  we have  $\mbox{dim}[\mbox{ker}(B^\top)]=0$, since  $B\neq 0$.
   From this, we can use $(iii)$ and $(iv)$ in Proposition \ref{pontryagin-maximum-principle} to
   see that $\mathbb{S}_{x_0}=\emptyset$.

   Thus, we complete the proof of Theorem \ref{theorem4.4,10-3}.
\end{proof}

\section{Examples}\label{example}

\subsection{Some examples on Theorem \ref{yu-theorem-9-19-1}}\label{yu-section-5}
    \textbf{Example on $(\textbf{A}_1)$.}
      {\it We will construct a time optimal control problem $(\mathcal{TP})_{x_0}$ so that its optimal control holds $(\textbf{A}_1)$.}
      This problem is given in the manner: $n=m\in\mathbb{N}^+$ is arbitrarily given;  $x_0\in\mathbb{R}^n\backslash\{0\}$ is arbitrarily given;
    $A\triangleq 0_{n\times n}$ (which is in $\mathbb{R}^{n\times n}$ whose all elements are zero) and $B\triangleq\mathbb{I}_{n\times n}$ (which is the identity matrix in $\mathbb{R}^{n\times n}$).
    First of all, one can easily see that  $\mbox{rank}(B,AB,\cdots,A^{n-1}B)=n$ and $\sigma(A)=\{0\}$.
     Then, according to notes $(\textbf{b}_3)$, $(\textbf{b}_2)$ and $(\textbf{b}_1)$ in Section \ref{introduce},
    this problem  has a unique optimal control $u^*_{x_0}$.

We now claim   that the optimal time and the optimal control to this problem are as:
\begin{equation}\label{yu-10-22-5}
    T^*_{x_0}=\|x_0\|_{\mathbb{R}^n},\;\;u^*_{x_0}(t)=
\begin{cases}
    -\frac{x_0}{\|x_0\|_{\mathbb{R}^n}},&t\in(0,T^*_{x_0}],\\
    0,&t\in (T^*_{x_0},+\infty).
\end{cases}
\end{equation}
    Indeed, since $0=x(T^*_{x_0};x_0,u^*_{x_0})=x_0+\int_0^{T^*_{x_0}}u^*_{x_0}(t)dt$, we have
    \begin{eqnarray}\label{3.1,10-26}
    \|x_0\|_{\mathbb{R}^n}\leq \int^{T^*_{x_0}}_0\|u^*_{x_0}(t)\|_{\mathbb{R}^n}dt
    \leq \int^{T^*_{x_0}}_0dt
    =T^*_{x_0}.
    \end{eqnarray}
    Meanwhile,  if let
$$
  T\triangleq \|x_0\|_{\mathbb{R}^n}\;\;\mbox{and}\;\;
    \hat u(t)\triangleq
\begin{cases}
    -\frac{x_0}{\|x_0\|_{\mathbb{R}^n}},&t\in(0,T],\\
    0,&t\in(T,+\infty),
\end{cases}
$$
    then we have
    $x(T;x_0,\hat u)=0$, i.e., $\hat u$ is an admissible control and $T\geq T^*_{x_0}$. These, along with \eqref{3.1,10-26},
    lead to $T=T^*_{x_0}$ and $\hat u$ is an optimal control. Then, by the uniqueness of the optimal control, we obtain
    \eqref{yu-10-22-5}.

    Finally, by \eqref{yu-10-22-5}, we see that the optimal control in this problem holds $(\textbf{A}_1)$ in Theorem \ref{yu-theorem-9-19-1}.

\vskip 5pt
    \noindent\textbf{Example on $(\textbf{A}_2)$.}
   {\it We will construct a time optimal control problem $(\mathcal{TP})_{x_0}$ so that its optimal control holds $(\textbf{A}_2)$.}
    This problem is given in the manner:  $n=m=2$; $x_0\triangleq(x_{0,1},x_{0,2})^\top\in\mathbb{R}^2\backslash\{0\}$ is arbitrarily given;
\begin{equation*}
    A\triangleq\left(
               \begin{array}{cc}
                 0 & 1 \\
                 -1 & 0 \\
               \end{array}
             \right)\;\;\mbox{and}\;\;B\triangleq\mathbb{I}_{2\times2}.
\end{equation*}
One can directly check that $\mbox{rank}(B,AB)=2$ and $\sigma(A)=\{i,-i\}$.  Then, according to notes $(\textbf{b}_3)$, $(\textbf{b}_2)$ and $(\textbf{b}_1)$ in Section \ref{introduce},
    this problem  has a unique optimal control $u^*_{x_0}$.
Also, one can easily verify
\begin{equation}\label{yu-9-2-1-bb}
    e^{At}=\left(
                   \begin{array}{cc}
                     \cos t & \sin t \\
                     -\sin t & \cos t \\
                   \end{array}
                 \right),\;\;t\in \mathbb{R}^+.
\end{equation}
By the same way to show \eqref{3.1,10-26}, we can use \eqref{yu-9-2-1-bb} to get
\begin{equation}\label{yu-10-3-50}
    T^*_{x_0}\geq \|x_0\|_{\mathbb{R}^2}.
\end{equation}

 We now let
    \begin{equation}\label{3.5,10-26}
    \widehat{T} \triangleq \|x_0\|_{\mathbb{R}^2}\;\;\mbox{and}\;\;\hat u(t)\triangleq
\begin{cases}
    -\widehat{T}^{-1}e^{At}x_0,&t\in(0,\widehat{T}],\\
    0,&t\in(\widehat{T},+\infty).
\end{cases}
\end{equation}
   By  \eqref{3.5,10-26} and \eqref{yu-9-2-1-bb}, we have
\begin{equation}\label{yu-9-2-1}
    \hat u(t)
    =
\begin{cases}
    -\widehat{T}^{-1}\left(
                          \begin{array}{cc}
                            \cos t & \sin t \\
                            -\sin  & \cos t \\
                          \end{array}
                        \right)
                        \left(
                                 \begin{array}{c}
                                   x_{0,1} \\
                                   x_{0,2} \\
                                 \end{array}
                               \right),&t\in(0,\widehat{T}],\\
    0,&t\in(\widehat{T},+\infty).
\end{cases}
\end{equation}
  From \eqref{yu-9-2-1}, we find that  $\hat u\in\mathcal{PC}(\mathbb{R}^+;B_1(0))$  and that
$\int_0^{\widehat{T}}e^{-At}\hat u(t)dt=-x_0$, i.e.,
   $x(\widehat{T};x_0,\hat u)=0$. Thus, $\hat u$ is an admissible control
   and    $\widehat T\geq T^*_{x_0}$.
    These,   along with \eqref{3.5,10-26}$_1$ and \eqref{yu-10-3-50}, yields that
     $\widehat T=T^*_{x_0}$ and $\hat u$ is an optimal control. Then by the uniqueness of the optimal control to this problem,
     we see that $u^*_{x_0}=\hat u$. From this and \eqref{yu-9-2-1}, we  see that the optimal control in this problem holds $(\textbf{A}_2)$ in Theorem \ref{yu-theorem-9-19-1}.

\vskip 10pt
    \noindent\textbf{Example on $(\textbf{A}_3)$.}
    {\it We will construct a time optimal control problem $(\mathcal{TP})_{x_0}$ so that its optimal control holds $(\textbf{A}_3)$.}
    This problem is given in the manner: $n=3$ and $m=2$; $x_0\triangleq(0,x_{0,2},x_{0,3})^\top\in\mathbb{R}^3$ with
\begin{equation}\label{yu-10-28-2}
    x_{0,2}^2+x_{0,3}^2\geq 4\pi^2;
\end{equation}
\begin{equation}\label{yu-10-28-7}
    A\triangleq\left(
         \begin{array}{ccc}
           -1 & 0 & 0 \\
           0 & 0 & 1 \\
           0 & -1 & 0 \\
         \end{array}
       \right),\;\;
       B\triangleq\left(
           \begin{array}{cc}
              1 & 0 \\
              0 & 0 \\
              0 & 1 \\
           \end{array}
         \right).
\end{equation}
    One can directly check that $\mbox{rank}(B,AB,A^2B)=3$ and
    $\sigma(A)=\{-1,i,-i\}$. Thus,
    according to notes $(\textbf{b}_3)$, $(\textbf{b}_2)$ and $(\textbf{b}_1)$ in Section \ref{introduce}, this problem has a unique optimal control $u^*_{x_0}$.
    %Moreover, one can easily prove that
%\begin{equation*}
%    e^{At}=\left(
%             \begin{array}{ccc}
%               e^{-t} & 0 & 0 \\
%               0 & \cos t & \sin t \\
%               0 & -\sin t & \cos t \\
%             \end{array}
%           \right),\;\;t\in\mathbb{R}^+.
%\end{equation*}
\par

    {\it We will show that $u^*_{x_0}$ holds $(\textbf{A}_3)$ in Theorem \ref{yu-theorem-9-19-1}.}
    This will be done by several steps.
\vskip 5pt
    \noindent{\it Step 1. We define a new time optimal control, denoted by   $(\widehat{\mathcal{TP}})_{\hat{x}_0}$.}

    It is defined
         in the following manner: $n\triangleq\hat{n}=2$ and $m\triangleq\hat{m}=1$; $x_0\triangleq \hat{x}_0\triangleq (x_{0,2},x_{0,3})^\top$;
    \begin{equation}\label{yu-10-28-9}
    \hat{A}\triangleq\left(
       \begin{array}{cc}
         0 & 1 \\
         -1 & 0 \\
       \end{array}
     \right),\;\;\hat{B}\triangleq\left(
       \begin{array}{c}
         0 \\
         1 \\
       \end{array}
     \right).
\end{equation}
  From  \eqref{yu-10-28-9}, we can see that
   $\mbox{rank}(\hat{B},\hat{A}\hat{B})=2$ and $\sigma(\hat{A})=\{i,-i\}$. These, along with
   notes $(\textbf{b}_3)$, $(\textbf{b}_2)$ and $(\textbf{b}_1)$ in Section \ref{introduce},
    show that
   $(\widehat{\mathcal{TP}})_{\hat{x}_0}$
    has a unique optimal control $\hat{u}^*_{\hat{x}_0}$.
     From  \eqref{yu-10-28-9}, {\it we also see that $e^{\hat{A}t}$ has the expression (\ref{yu-9-2-1-bb}).}

\vskip 5pt
    \noindent{\it Step 2. We prove that $\hat{u}^*_{\hat{x}_0}$ is a non-constant valued step function taking only two values: $1$ and $-1$.}

    Write $\widehat{T}^*_{\hat{x}_0}$ for the optimal time of $(\widehat{\mathcal{TP}})_{\hat{x}_0}$.
    Denote by $\widehat{\mathbb{S}}_{\hat{x}_0}$ the set of all switching points of  $\hat{u}^*_{\hat{x}_0}$
    over $(0,\widehat{T}^*_{\hat{x}_0})$.
    Let $\hat{z}^*\triangleq(\hat{z}^*_1,\hat{z}^*_2)^{\top}\in\mathbb{R}^2\setminus\{0\}$  be given by the conclusion $(iii)$ of  Proposition \ref{pontryagin-maximum-principle}
    (where $(\mathcal{TP})_{x_0}$ is replaced by $(\widehat{\mathcal{TP}})_{\hat{x}_0}$).

    First, by \eqref{yu-10-28-9} and  (\ref{yu-9-2-1-bb}) (where $A=\hat A$), we have
    \begin{eqnarray}\label{3.10,10-29}
    \hat{B}^\top e^{\hat{A}^\top(\widehat{T}^*_{\hat{x}_0}-t)}\hat{z}^*
    =\|\hat{z}^*\|_{\mathbb{R}^2}\sin(\widehat{T}^*_{\hat{x}_0}+\theta-t),\;\;t\in \mathbb{R}^+,
    \end{eqnarray}
    where $\theta\triangleq\arctan (\hat{z}^*_2/\hat{z}^*_1)$. (Here, we agree that $\arctan(\infty)=\pi/2$ and $\arctan(-\infty)=-\pi/2$.)
   Second, by the same way to show \eqref{yu-10-3-50}, we can use (\ref{yu-9-2-1-bb}) (where $A=\hat A$) and (\ref{yu-10-28-2})
   to see
   \begin{equation}\label{yu-10-28-1}
    \widehat{T}^*_{\hat{x}_0}\geq \|\hat{x}_0\|_{\mathbb{R}^2}\geq 2\pi.
\end{equation}
    Then by  $(iv)$ in Proposition \ref{pontryagin-maximum-principle} and \eqref{3.10,10-29}, we find
   \begin{eqnarray}\label{yu-10-28-3}
    \hat{u}^*_{\hat{x}_0}(t)
    =
\begin{cases}
    \frac{\sin(\widehat{T}^*_{\hat{x}_0}+\theta-t)}{|\sin(\widehat{T}^*_{\hat{x}_0}+\theta-t)|},&
    t\in(0,\widehat{T}^*_{\hat{x}_0})\setminus \widehat{\mathcal{O}}_{\hat{z}^*},\\
    \lim_{s\to t^-}\frac{\sin(\widehat{T}^*_{\hat{x}_0}+\theta-s)}{|\sin(\widehat{T}^*_{\hat{x}_0}+\theta-s)|},
    &t\in\widehat{\mathcal{O}}_{\hat{z}^*},
\end{cases}
\end{eqnarray}
    where $\widehat{\mathcal{O}}_{\hat{z}^*}\triangleq\{t\in(0,\widehat{T}^*_{\hat{x}_0}):
    \sin(\widehat{T}^*_{\hat{x}_0}+\theta-t)=0\}$.
    Finally, by \eqref{yu-10-28-3} and the property of the function: $\sin t$ ($t\in \mathbb{R}$), yields
   \begin{equation}\label{3.13,10-29}
    \widehat{\mathbb{S}}_{\hat{x}_0}=\widehat{\mathcal{O}}_{\hat{z}^*}
    =\{\widehat{T}^*_{\hat{x}_0}+\theta-n\pi:n\in\mathbb{N}
    \;\;\mbox{s.t.}\;\;n\pi\in(\theta,\widehat{T}^*_{\hat{x}_0}+\theta)\},
\end{equation}
   which, along with (\ref{yu-10-28-1}), leads to  $\widehat{\mathbb{S}}_{\hat{x}_0}\neq\emptyset$.
   This, together with (\ref{yu-10-28-3}) and the first equality in \eqref{3.13,10-29}, shows the conclusion in {\it Step 2}.
\vskip 5pt
   \noindent{\it Step 3. We prove $u^*_{x_0}=(0,\hat{u}^*_{\hat{x}_0})^{\top}$.}
\par

     Write $x(\cdot;x_0,u)$  for the solution to the equation:
     $x'(t)=Ax(t)+Bu(t), \;t\geq 0$; $x(0)=x_0$. Write $\hat x(\cdot;\hat x_0,\hat u)$ for the solution to the equation:
     $\hat x'(t)=\hat A\hat x(t)+\hat B\hat u(t), \;t\geq 0$; $\hat x(0)=\hat x_0$. Then, from \eqref{yu-10-28-7}
     and \eqref{yu-10-28-9}, using the fact that $x_0=(0,\hat x_0^{\top})^{\top}$, one can easily find for each $u=(u_1,u_2)^{\top}$,
     \begin{eqnarray*}
     x(t;x_0,(u_1,u_2)^{\top})=(\tilde{x}(t;u_1),\hat x(t;\hat x_0, u_2))^\top\;\;\mbox{for all}\;\;t\geq 0,
     \end{eqnarray*}
    where $\tilde{x}(t;u_1)=\int_0^te^{-(t-s)}u_1(s)ds$. From this and the optimality of $u^*_{x_0}$ and $\hat{u}^*_{\hat{x}_0}$, as well as the uniqueness of the optimal control to $(\mathcal{TP})_{x_0}$,  we can directly verify that
     $u^*_{x_0}=(0,\hat{u}^*_{\hat{x}_0})^{\top}$.

     \vskip 5pt

   \noindent{\it Step 4. We prove that $u^*_{x_0}$  holds $(\textbf{A}_3)$ in Theorem \ref{yu-theorem-9-19-1}.}

   Indeed, by  conclusions in \emph{Steps 2-3}, we see that $u^*_{x_0}$ is a  non-constant valued
    step function taking only two values: $(0,1)^\top$ and $(0,-1)^\top$.

\vskip 5pt
    \noindent\textbf{Example on $(\textbf{A}_4)$.}
    {\it We will construct a time optimal control problem $(\mathcal{TP})_{x_0}$ so that its optimal control holds $(\textbf{A}_4)$.}
       First of all, we let $\xi\in\mathbb{R}\setminus\{0\}$ be fixed so that
\begin{equation}\label{yu-9-22-53}
    |\xi|\int_0^{4\pi}|\sin t|\sqrt{4\cos^2t+1}dt<1.
\end{equation}
    Then there is a unique time
    $\overline{T}>0$ so that
\begin{equation}\label{yu-9-22-54}
    |\xi|\int_0^{\overline{T}}|\sin t|\sqrt{4\cos^2t+1}dt=1.
\end{equation}
   This, along with  (\ref{yu-9-22-53}), yields that
\begin{equation}\label{yu-10-18-11}
    \overline{T}>4\pi.
\end{equation}
\par
  We now define the time optimal control problem $(\mathcal{TP})_{x_0}$  in the manner: $n=4$, $m=2$;
\begin{equation}\label{yu-9-22-50}
    A\triangleq\left(
                 \begin{array}{cccc}
                   0 & 1 & 0 & 0 \\
                   -1 & 0 & 0 & 0 \\
                   0 & 0 & 0 & 2 \\
                   0 & 0 & -2 & 0 \\
                 \end{array}
               \right),\;\;B\triangleq\left(
                                        \begin{array}{cc}
                                          0 & 0 \\
                                          1 & 0 \\
                                          0 & 0 \\
                                          0 & 1 \\
                                        \end{array}
                                      \right);
\end{equation}
    \begin{equation}\label{yu-9-22-55}
    x_0\triangleq\frac{\xi}{|\xi|}\int_0^{\overline{T}}
    \frac{\sin t}{|\sin t|\sqrt{4\cos^2t+1}}
    \left(
      \begin{array}{c}
        -\sin(2t) \\
        2\cos^2t \\
        -\sin(2t) \\
        \cos(2t) \\
      \end{array}
    \right)dt.
\end{equation}
   From \eqref{yu-9-22-50}, one has that
   \begin{equation}\label{yu-9-22-51}
    e^{At}=\left(
             \begin{array}{cccc}
               \cos t & \sin t & 0 & 0 \\
               -\sin t & \cos t & 0 & 0 \\
               0 & 0 & \cos(2t) & \sin(2t) \\
               0 & 0 & -\sin(2t) & \cos(2t) \\
             \end{array}
           \right),\;\;t\in\mathbb{R}
\end{equation}
   and that  $\mbox{rank}(B,AB,A^2B,A^3B)=4$ and $\sigma(A)=\{i,-i,2i,-2i\}$.
   The later two facts, along with the notes $(\textbf{b}_3)$, $(\textbf{b}_2)$ and $(\textbf{b}_1)$ in Section \ref{introduce}, yield that  $(\mathcal{TP})_{x_0}$ has a unique optimal control $u^*_{x_0}$, which implies
    \begin{equation}\label{yu-9-22-103}
    x_0=-\int_0^{T^*_{x_0}}e^{-At}Bu^*_{x_0}(t)dt.
\end{equation}
    Let
     $\hat{z}\triangleq(\xi,0,\xi,0)^\top \in\mathbb{R}^4$. Then by (\ref{yu-9-22-51}) and
     \eqref{yu-9-22-50}$_2$,  we have
\begin{equation}\label{yu-9-22-52}
    B^\top e^{-A^\top t}\hat{z}=-\xi\left(
                            \begin{array}{c}
                              \sin (2t) \\
                              \sin t \\
                            \end{array}
                          \right).
\end{equation}
    Define the following function:
   \begin{eqnarray}\label{yu-9-22-56}
    \bar{u}(t)\triangleq
    \begin{cases}
    -\frac{\sin t}{|\sin t|}\frac{\xi(2\cos t,1)^\top}{|\xi|\sqrt{4\cos ^2t+1}},&t\in(0,\overline{T})\setminus
    \{s\in\mathbb{R}^+:\sin s= 0\},\\
    -\lim_{s\to t^-}\frac{\sin s}{|\sin s|}\frac{\xi(2\cos t,1)^\top}{|\xi|\sqrt{4\cos ^2t+1}},&
    t\in\left[(0,\overline{T})\bigcap \{s\in\mathbb{R}^+:\sin s= 0\}\right]\bigcup\{\overline{T}\},\\
    0,&t\in(\overline{T},+\infty).
\end{cases}
\end{eqnarray}

\par

    We claim that
    \begin{equation}\label{yu-10-28-13}
T^*_{x_0}=\overline{T}\;\;\mbox{and}\;\;u^*_{x_0}=\bar{u}.
\end{equation}
    Indeed, by (\ref{yu-9-22-55}), (\ref{yu-9-22-51}) and (\ref{yu-9-22-56}),
    after some simple computations, we find
    \begin{equation}\label{yu-10-28-16}
    x_0=-\int_0^{\overline{T}}e^{-At}B\bar{u}(t)dt,\;\mbox{i.e.},\; x(\overline{T};x_0,\bar{u})=0,
\end{equation}
   which implies that $\overline{u}$  an admissible control and that $T^*_{x_0}\leq \overline{T}$.
    Meanwhile, by  (\ref{yu-9-22-56}) and (\ref{yu-9-22-52}), we have
\begin{equation*}
    \int_0^{\overline{T}}\langle \bar{u}(t),B^\top e^{-A^\top t}\hat{z}\rangle_{\mathbb{R}^2}dt=|\xi|\int_0^{\overline{T}}|\sin t|\sqrt{4\cos^2t+1}dt.
\end{equation*}
    This, together with (\ref{yu-10-28-16}) and (\ref{yu-9-22-54}), leads to
$\langle x_0,\hat{z}\rangle_{\mathbb{R}^4}=-1$,
which, along with  (\ref{yu-9-22-103}), implies
$$
    \int_0^{T^*_{x_0}}\|B^\top e^{-A^\top t}\hat{z}\|_{\mathbb{R}^2}dt\geq \int_0^{T^*_{x_0}}\langle u^*_{x_0}(t),B^\top e^{-A^\top t}\hat{z}\rangle_{\mathbb{R}^2}dt=-\langle x_0,\hat{z}\rangle_{\mathbb{R}^4}=1.
$$
   From the above,  (\ref{yu-9-22-54}) and the fact that $T^*_{x_0}\leq \overline{T}$, we obtain \eqref{yu-10-28-13}$_1$. Then by \eqref{yu-10-28-13}$_1$, \eqref{yu-10-28-16} and the uniqueness of the optimal control to $(\mathcal{TP})_{x_0}$, we get \eqref{yu-10-28-13}$_2$.
   Thus, we have proved \eqref{yu-10-28-13}.

   Next, by
  (\ref{yu-9-22-56}) and the property of the function: $\sin t$ $(t\in\mathbb{R})$,  we see that $\mathbb{S}_{x_0}=\{n\pi:n\pi\in(0,T^*_{x_0}),\;n\in\mathbb{N}\}$, which is not empty because of  (\ref{yu-10-18-11}).
   Meanwhile, by (\ref{yu-9-22-56}) and (\ref{yu-10-28-13})$_2$, we see that $u^*_{x_0}$ is not constant-valued in any subinterval of $(0,T^*_{x_0})$.
    From these,  we see that $u^*_{x_0}$ holds $(\textbf{A}_4)$ in Theorem \ref{yu-theorem-9-19-1}.

\subsection{An example on $\mathbb{S}_{x_0}\neq (T^*_{x_0}-\mathcal{O}_{z^*})\cap (0,T^*_{x_0})$}\label{subsection-11-8-1}
   We will construct a time optimal control problem  $(\mathcal{TP})_{x_0}$
   so that $\mathbb{S}_{x_0}$ is a proper subset of $(T^*_{x_0}-\mathcal{O}_{z^*})\cap (0,T^*_{x_0})$
   for some $z^*$ with \eqref{pontryagin-maximum-principle-1}. This problem is given in the manner:
   $n=4$ and $m=1$; $x_0\in\mathbb{R}^4\setminus\{0\}$ will be given later;
\begin{equation*}\label{yu-11-8-1}
    A\triangleq \left(
                  \begin{array}{cccc}
                    -1 & 0 & 1 & 1 \\
                    0 & -1 & 0 & 1 \\
                    0 & 0 & 0 & 1 \\
                    0 & 0 & -1 & 0 \\
                  \end{array}
                \right),\;\;B\triangleq\left(
                                         \begin{array}{c}
                                           0 \\
                                           0 \\
                                           0 \\
                                           1 \\
                                         \end{array}
                                       \right).
\end{equation*}
    One can directly check that
\begin{equation}\label{yu-11-8-1-b}
  \sigma(A)=\{-1,i,-i\}\;\;\mbox{and}\;\;  \mbox{rank}(B,AB,A^2B,A^3B)=4.
\end{equation}
  These, along with the notes $(\textbf{b}_3)$, $(\textbf{b}_2)$ and $(\textbf{b}_1)$ in Section \ref{introduce}, yields that for any $x_0\in\mathbb{R}^4\setminus\{0\}$, this problem has a unique optimal control.
   Also we can directly check that
    $\mbox{dim}\left[\mbox{span}\{B,AB,A^2B\}\bigcap \mbox{span}\{B,AB\}^\bot\right]=1$.
    Thus, we can take
\begin{equation}\label{yu-11-8-4}
    \hat{e}\in \mbox{span}\{B,AB,A^2B\}\bigcap \mbox{span}\{B,AB\}^\bot,\;\mbox{with}\; \|\hat{e}\|_{\mathbb{R}^4}=1.
\end{equation}
    We claim
\begin{equation}\label{yu-11-9-1}
    B\hat{e}=0,\;\;B^\top A^\top \hat{e}=0\;\;\mbox{and}\;\;B^\top (A^\top)^2\hat{e}\neq 0.
\end{equation}
    Indeed, by (\ref{yu-11-8-4}) and (\ref{yu-9-28-5}), we get (\ref{yu-11-9-1})$_1$ and (\ref{yu-11-9-1})$_2$ easily.
    Next, if $B^\top(A^\top)^2\hat{e}= 0$, then, we see from (\ref{yu-11-9-1})$_1$, (\ref{yu-11-9-1})$_2$ and (\ref{yu-9-28-5}) that
$$
    \hat{e}\in\bigcap_{j=0}^2\mbox{ker}(B^\top (A^\top)^j)= \mbox{span}\{B,AB,A^2B\}^\bot.
$$
    This, along with (\ref{yu-11-8-4}), yields $\hat{e}=0$, which leads to a contradiction.  Thus (\ref{yu-11-9-1})$_3$ holds.

     Arbitrarily fix $\hat t>0$. Let
\begin{equation}\label{yu-11-8-5}
    \hat{f}\triangleq e^{A^\top \hat{t}}\hat{e}.
\end{equation}
     Then by (\ref{yu-11-8-5}) and (\ref{yu-11-9-1}), we see
\begin{equation}\label{yu-11-8-6}
    B^\top e^{-A^\top \hat{t}}\hat{f}=0,\;\;B^\top A^{\top} e^{-A^\top \hat{t}}\hat{f}=0
    \;\;\mbox{and}\;\;B^\top (A^{\top})^2 e^{-A^\top \hat{t}}\hat{f}\neq 0.
\end{equation}
Take $\xi>0$ so that
\begin{equation}\label{yu-11-8-7}
    \xi\int_0^{\hat{t}}\|B^\top e^{-A^\top t}\hat{f}\|_{\mathbb{R}^1}dt<1.
\end{equation}
    By  (\ref{yu-11-8-1-b}), we can use
    \cite[pp.159 and pp.168]{Conti}  (or \cite[(2.10) in the proof of Theorem 2.6, pp.23]{Evans}) to find  a unique  $\overline{T}>\hat{t}$ so that
\begin{equation}\label{yu-11-8-8}
    \xi\int_0^{\overline{T}}\|B^\top e^{-A^\top t}\hat{f}\|_{\mathbb{R}^1}dt=1.
\end{equation}
    Let
    \begin{equation}\label{yu-11-8-9}
    \hat{z}\triangleq\xi e^{-A^\top \overline{T}}\hat{f},
\end{equation}
\begin{equation}\label{yu-11-8-10}
    \bar{u}(t)\triangleq
\begin{cases}
    \frac{B^\top e^{A^\top (\overline{T}-t)}\hat{z}}{\|B^\top e^{A^\top (\overline{T}-t)}\hat{z}\|_{\mathbb{R}^1}},&t\in(0,\overline{T})\setminus (\overline{T}-\mathcal{O}_{\hat{z}}),\\
    \lim_{s\to t}\frac{B^\top e^{A^\top (\overline{T}-s)}\hat{z}}{\|B^\top e^{A^\top (\overline{T}-s)}\hat{z}\|_{\mathbb{R}^1}},&t\in[(\overline{T}-\mathcal{O}_{\hat{z}})\bigcap(0,\overline{T})]
    \bigcup\{\overline{T}\},\\
    0,&t\in(\overline{T},+\infty)
\end{cases}
\end{equation}
    and
\begin{equation}\label{yu-11-8-11}
    x_0\triangleq-\int_0^{\overline{T}}e^{-At}B\bar{u}(t)dt,
\end{equation}
    where $\mathcal{O}_{\hat{z}}$ is given by (\ref{main-theorem-lemma1-1}) with $z=\hat{z}$. By \eqref{yu-11-8-11}, (\ref{yu-11-8-10}), (\ref{yu-11-8-9}) and (\ref{yu-11-8-8}), one can easily check that $\langle x_0,\hat{z}\rangle_{\mathbb{R}^4}=-1$ and then $x_0\neq 0$.

     Let  $u^*_{x_0}$ be the optimal control to $(\mathcal{TP})_{x_0}$ with $x_0$ given by \eqref{yu-11-8-11}. We claim
\begin{equation}\label{yu-11-8-26}
    \hat{z}\in\mathbb{R}^4\setminus\{0\}\;\;\mbox{is a multiplier so that (\ref{pontryagin-maximum-principle-1}) holds},
\end{equation}
     and
\begin{equation}\label{yu-11-8-25}
    \hat{t}\in\left[(T^*_{x_0}-\mathcal{O}_{\hat{z}})\bigcap (0,T^*_{x_0})\right]\setminus\mathbb{S}_{x_0}.
\end{equation}
    When these are proved,   the aim of this example is reached.
  The proofs of (\ref{yu-11-8-26}) and (\ref{yu-11-8-25}) are divided into several steps.
\par
    \noindent\emph{Step 1. We prove}
\begin{equation}\label{yu-11-8-12}
    T^*_{x_0}=\overline{T}\;\;\mbox{and}\;\;u^*_{x_0}=\bar{u}.
\end{equation}
\par
    By (\ref{yu-11-8-11}), (\ref{yu-11-8-10}), (\ref{yu-11-8-9}) and (\ref{yu-11-8-8}),
     using a very similar method to that used in the proof of (\ref{yu-10-28-13}), we can get (\ref{yu-11-8-12}).

\par
    \noindent\emph{Step 2. We prove (\ref{yu-11-8-26}).}
\par
    By (\ref{yu-11-8-9}) and (\ref{yu-11-8-6})$_3$, we see that  $\hat{z}\neq 0$. Meanwhile, it follows  by (\ref{yu-11-8-12}) and (\ref{yu-11-8-10}) that
\begin{equation*}\label{yu-11-8-16}
    \langle u^*_{x_0}(t),B^\top e^{A^\top(T^*_{x_0}-t)}\hat{z}\rangle_{\mathbb{R}^1}
    =\|B^\top e^{A^\top(T^*_{x_0}-t)}\hat{z}\|_{\mathbb{R}^1}=\max_{u\in B_1(0)}\langle u,B^\top e^{A^\top(T^*_{x_0}-t)}\hat{z}\rangle_{\mathbb{R}^1}\;\;\mbox{a.e.}\;\;t\in(0,T^*_{x_0}).
\end{equation*}
    Thus (\ref{yu-11-8-26}) holds.

\par
    \noindent\emph{Step 3. We prove $\hat{t}\in (T^*_{x_0}-\mathcal{O}_{\hat{z}})\cap (0,T^*_{x_0})$.}
\par
    First, it follows  by (\ref{yu-11-8-7}), (\ref{yu-11-8-8}) and (\ref{yu-11-8-12})$_1$ that $\hat{t}\in (0,T^*_{x_0})$. Second, it follows from  (\ref{yu-11-8-12})$_1$, (\ref{yu-11-8-9}), (\ref{yu-11-8-5}) and (\ref{yu-11-9-1})$_1$ that
\begin{equation*}\label{yu-11-8-17}
    B^\top e^{A^\top (T^*_{x_0}-\hat{t})}\hat{z}=\xi B^\top e^{-A^\top \hat{t}}\hat{f}=\xi B^\top \hat{e}=0,
\end{equation*}
which, along with (\ref{main-theorem-lemma1-1}) (where $z=\hat{z}$), yields that $\hat{t}\in T^*_{x_0}-\mathcal{O}_{\hat{z}}$.
\par
    \noindent\emph{Step 4. We prove (\ref{yu-11-8-25}).}
\par
    By the conclusions in \emph{Steps 1-3},
    we see that in order to have (\ref{yu-11-8-25}), it suffices to show $ \hat{t}\notin\mathbb{S}_{x_0}$.
    For this purpose, we use  (\ref{yu-11-8-9}) and (\ref{yu-11-8-6}) to get
\begin{equation*}\label{yu-11-8-19}
    B^\top e^{A^\top (T^*_{x_0}-t)}\hat{z}=\xi B^\top e^{-A^\top t}\hat{f}
    =\xi(t-\hat{t})^2[a_2+b_2(t)(t-\hat{t})],
\end{equation*}
    where $a_2\triangleq B^\top (-A^\top)^2e^{-A^\top\hat{t}}\hat{f}$  and $b_3(t)\triangleq\sum_{j=3}^{+\infty}B^\top(-A^\top)^je^{-A^\top \hat{t}}\hat{f}(t-\hat{t})^{j-3}$.
    This, together with (\ref{yu-11-8-12})$_2$,  (\ref{yu-11-8-10}) and (\ref{yu-11-8-6})$_3$, implies that
$$
    \lim_{t\to\hat{t}^+}u^*_{x_0}(t)=\lim_{t\to\hat{t}^-}u^*_{x_0}(t)=\frac{a_2}{\|a_2\|_{\mathbb{R}^1}}.
$$
   This, along with  (\ref{yu-9-27-10}), shows that $\hat{t}\notin \mathbb{S}_{x_0}$.

\section{Further information  on switching directions}\label{appendix-9}

 In this section, we will present some further information on  $\mathbb{D}_{x_0}$ (given by
\eqref{yu-9-27-10}) for some special cases.

\begin{theorem}\label{yu-proposition-10-16-1}
     Assume that $(\textbf{H})_{x_0}$ holds. If $n\geq 2$ and  $\mathbb{S}_{x_0}\neq \emptyset$,  then the following two conclusions are true:
\begin{description}
  \item[$(i)$]  If $\mbox{rank}(B)=1$, then
  $\mathbb{D}_{x_0}$ contains at most two vectors:
  $\frac{B^\top v}{\|B^\top v\|_{\mathbb{R}^m}}$ and $-\frac{B^\top v}{\|B^\top v\|_{\mathbb{R}^m}}$,
  where $v$ is a unit vector in the $1$-dim space: $\mbox{ker}(B^\top)^\bot$.
  \item[$(ii)$] If $\mbox{rank}(B,AB,\cdots,A^{n-1}B)=n$ and $\mbox{rank}(B)=n-1$, then $\mathbb{D}_{x_0}$ contains at most two vectors:
  $\frac{B^\top A^\top v}{\|B^\top A^\top v\|_{\mathbb{R}^m}}$ and $-\frac{B^\top A^\top v}{\|B^\top A^\top v\|_{\mathbb{R}^m}}$, where $v$ is a unit vector in the $1$-dim space: $\mbox{ker}(B^\top)$.

\end{description}
\end{theorem}
    Several notes on Theorem \ref{yu-proposition-10-16-1} are given in order:
\begin{enumerate}
  \item[($\textbf{e}_1$)] In the case that $n\geq 2$, $\mbox{rank}(B)=1$, $\mathbb{S}_{x_0}\neq \emptyset$ and
  $(\textbf{H})_{x_0}$ holds, we have the following conclusions:
\begin{itemize}
\item  Any  optimal control satisfies   $(\textbf{A}_3)$ in Theorem \ref{yu-theorem-9-19-1}.
\item When $\sharp \mathbb{S}_{x_0}=1$, $\mathbb{D}_{x_0}$ is a singleton set, while when
    $\sharp \mathbb{S}_{x_0}\geq 2$, $\mathbb{D}_{x_0}$ consists of two opposite vectors.
\end{itemize}
The first conclusion above can be seen from the proof of  $(i)$ in Theorem \ref{yu-proposition-10-16-1},
  while the second one  follows from  $(i)$ of Theorem \ref{yu-proposition-10-16-1}.

  \item[($\textbf{e}_2$)] In the case that $n=3$,  $\mbox{rank}(B)=2$, $\mathbb{S}_{x_0}\neq \emptyset$ and
  $(\textbf{H})_{x_0}$ holds,  the assumption $\mbox{rank}(B,AB,A^2B)=3$ in $(ii)$ of Theorem \ref{yu-proposition-10-16-1} can be dropped.

  Indeed, in this case, the conclusions (\ref{yu-10-19-3}) and (\ref{yu-10-27-3}) in \emph{Step 1} of the proof of Theorem \ref{theorem4.4,10-3} are true. Thus, we have $k(\hat{t};\hat{z})=1$, where $k(\hat{t};\hat{z})$ is defined by (\ref{yu-10-19-3-b}).
  Then, by the same way as that used in the proof of \emph{Step 2} in $(ii)$ of  Theorem \ref{theorem4.4,10-3}, we can show that $\mathbb{D}_{x_0}$ has at most two directions. (We omit the details here.)

\item[($\textbf{e}_3$)]
  In the case that $n\geq 2$, $\mbox{rank}(B,AB,\cdots,A^{n-1}B)=n$, $\mbox{rank}(B)=n-1$, $\mathbb{S}_{x_0}\neq \emptyset$ and
  $(\textbf{H})_{x_0}$ holds, we have $B^\top A^\top v\neq 0$ for any unit vector $v\in \mbox{ker}(B^\top)$.
  This can be seen from (\ref{yu-11-10-11})$_2$ in the proof of $(ii)$ of Theorem \ref{yu-proposition-10-16-1}.
\end{enumerate}

\begin{proof}[Proof of Theorem \ref{yu-proposition-10-16-1}]
Let $u^*_{x_0}$ be the optimal control to $(\mathcal{TP})_{x_0}$.
Let  $z^*\in\mathbb{R}^n\setminus\{0\}$ be given by  $(iii)$ of Proposition \ref{pontryagin-maximum-principle} and let   $\hat{z}^*$ be given by (\ref{new2.24,10-22}).
 Since $z^*\neq 0$, it follows by (\ref{new2.24,10-22}) that
 \begin{eqnarray}\label{NEW2.24,10-25}
e^{-A^\top t}\hat z^*\neq 0\;\;\mbox{for each}\;\;t\in \mathbb{R}.
 \end{eqnarray}
     Arbitrarily fix $\hat{t}\in\mathbb{S}_{x_0}$. Let $k(\hat{t};\hat{z}^*)$ be given by (\ref{yu-10-19-3-b}).

\vskip 5pt

{\it We first prove $(i)$ of Theorem \ref{yu-proposition-10-16-1}.}
 Since $\mbox{rank}(B)=1$, we have
     \begin{eqnarray}\label{new2.25-10-22}
 \mbox{dim}[\mbox{ker}(B^\top)]=(n-1)\;\;\mbox{and}\;\;
 \mbox{dim}[\mbox{ker}(B^\top)]^\bot=1.
 \end{eqnarray}
    By \eqref{new2.25-10-22}, we can take an orthonormal basis $\{e_j\}_{j=1}^n$ of $\mathbb{R}^n$ so that
\begin{eqnarray}\label{2.24,10-22}
 e_1, e_2,\ldots, e_{n-1}\in \mbox{ker}(B^\top)\;\;\mbox{and}\;\;e_n\in [\mbox{ker}(B^\top)]^\bot.
 \end{eqnarray}
            For each $i=1,2,\ldots,n$, we define a function:
      \begin{eqnarray}\label{yu-9-15-2}
      f_i(t)\triangleq \langle e^{-A^\top t}\hat{z}^*,e_i\rangle_{\mathbb{R}^n},\;\;t\in \mathbb{R},
      \end{eqnarray}
    which  is  real analytical. By (\ref{yu-9-15-2}), we have
\begin{equation}\label{yu-9-15-4}
    e^{-A^\top t}\hat{z}^*= \sum_{j=1}^nf_j(t)e_j
    \;\;\mbox{for all}\;\;t\in\mathbb{R}^+.
\end{equation}
    By (\ref{yu-9-15-4}) and \eqref{2.24,10-22}, we have
    \begin{equation}\label{yu-9-15-5}
    B^\top e^{-A^\top t}\hat{z}^*=f_n(t)B^\top e_n\;\;\mbox{for all}\;\; t\in (0,T^*_{x_0}).
\end{equation}
  Meanwhile, by \eqref{2.24,10-22} and \eqref{new2.24,10-22}, we see
     \begin{eqnarray}\label{2.29,10-22}
   B^\top e_n\neq 0\;\;\mbox{and}\;\; B^\top e^{-A^\top t}\hat{z}^*=B^\top e^{A^\top (T^*-t)}z^*\;\;\mbox{for all}\;\;t\in (0,T^*_{x_0}).
   \end{eqnarray}
     From \eqref{main-theorem-lemma1-1}, \eqref{yu-9-15-5} and \eqref{2.29,10-22}, we find
   \begin{eqnarray}\label{2.30,10-22}
   \mathcal{O}_{z^*}=T^*_{x_0}-\{s\in \mathbb{R}: f_n(s)=0\}\;\;\mbox{and}\;\;B^\top e^{A^\top (T^*_{x_0}-t)}z^*= f_n(t)B^\top e_n\;\;\mbox{for each}\;t\in (0,T^*_{x_0}).
   \end{eqnarray}
   Finally, by \eqref{2.30,10-22}, we can use the same way as that used in the proof of \eqref{wang2.19}
   (there, the analyticity of $f_i$ was used)
      to see that for each $t\in (0,T^*_{x_0})$,
   \begin{eqnarray*}
    u^*_{x_0}(t)=\pm \frac{B^\top e_n}{\|B^\top e_n\|_{\mathbb{R}^m}}.
   \end{eqnarray*}
     Since $\mathbb{S}_{x_0}\neq\emptyset$, the above yields that $u^*_{x_0}(\cdot)$ is a non-constant valued step function
  taking  two values: $\frac{B^\top e_n}{\|B^\top e_n\|_{\mathbb{R}^m}}$ and  $-\frac{B^\top e_n}{\|B^\top e_n\|_{\mathbb{R}^m}}$.
    This, along with \eqref{yu-9-27-10} (which requires $\mathbb{S}_{x_0}\neq \emptyset$),
    leads to what follows: when $\mathbb{S}_{x_0}$ contains only one point,
    $\mathbb{D}_{x_0}$ contains only one direction (which is either $\frac{B^\top e_n}{\|B^\top e_n\|_{\mathbb{R}^m}}$ or  $-\frac{B^\top e_n}{\|B^\top e_n\|_{\mathbb{R}^m}}$), while when $\mathbb{S}_{x_0}$ contains more than one point,
    $\mathbb{D}_{x_0}=\left\{\frac{B^\top e_n}{\|B^\top e_n\|_{\mathbb{R}^m}}, -\frac{B^\top e_n}{\|B^\top e_n\|_{\mathbb{R}^m}}\right\}$.
    Thus,  we end the proof of the conclusion $(i)$ of Theorem \ref{yu-proposition-10-16-1}.
\vskip 5pt

 {\it We next show  $(ii)$ of Theorem \ref{yu-proposition-10-16-1}.}
 Suppose
 \begin{eqnarray}\label{4.9.11-14}
 \mbox{rank}(B,AB,\ldots,A^{n-1}B)=n\;\;\mbox{and}\;\;\mbox{rank}(B)=(n-1).
 \end{eqnarray}
 By \eqref{4.9.11-14}$_2$, we have
    \begin{eqnarray}\label{2.31,10-25}
    \mbox{dim}[\mbox{ker}(B^\top)]=1,
    \;\;\;\mbox{dim}[\mbox{ker}(B^\top)]^\bot=(n-1),\;\;\mbox{dim}[\mbox{span}\{B\}]=(n-1).
    \end{eqnarray}
    The rest of the proof is organized by two steps.
 \vskip 5pt
   \noindent {\it Step 1. We prove that $k(\hat{t};\hat{z}^*)=1$.}
\par
     We first claim that
\begin{equation}\label{yu-11-10-1}
    \mbox{dim}[\mbox{span}\{B,AB\}]=n.
\end{equation}
  By contradiction, we suppose that  \eqref{yu-11-10-1} is not true.
  Then we have
\begin{equation}\label{yu-11-10-10}
    \mbox{dim}[\mbox{span}\{B,AB\}]<n.
\end{equation}
     Meanwhile, by \eqref{2.31,10-25}$_3$, we have $\mbox{dim}[\mbox{span}\{B,AB\}]\geq (n-1)$. This, along with (\ref{yu-11-10-10}), yields that $\mbox{dim}[\mbox{span}\{B,AB\}]=(n-1)$.
     By this and by making use of \eqref{2.31,10-25}$_3$ again, we find
    $\mbox{span}\{B\}=\mbox{span}\{B,AB\}$.
    Thus we have
    $$
    \mbox{span}\{AB\}\subset \mbox{span}\{B\}.
    $$
    By the above, we see that, for each $u\in \mathbb{R}^m$, there
    are $\{v_j\}_{j=1}^{n-1}\subset\mathbb{R}^m$ so that
\begin{equation*}
    ABu=Bv_1,\;\;A^2Bu(=A(ABu)=ABv_1)=Bv_2,\;\;\dots,\;\;A^{n-1}Bu=Bv_{n-1}.
\end{equation*}
      The above leads to $\mbox{span}\{B,AB,\ldots, A^{n-1}B\}\subset \mbox{span}\{B\}$
    and then
$\mbox{rank}(B,AB,\cdots, A^{n-1}B)\leq (n-1)$,
    which contradicts \eqref{4.9.11-14}$_1$. So  (\ref{yu-11-10-1}) is true.
\par
    We now show $k(\hat{t};\hat{z}^*)=1$. By (\ref{yu-10-19-3-b}), (\ref{new2.24,10-22}) and the fact $\hat{t}\in\mathbb{S}_{x_0}$, we see that $k(\hat{t};\hat{z}^*)\geq 1$. Suppose that $k(\hat{t};\hat{z}^*)>1$. Then by (\ref{yu-10-19-3-b}), we have $B^\top e^{-A^\top\hat{t}}\hat{z}^*=0$ and $B^{\top}A^\top e^{-A^\top\hat{t}}\hat{z}^*=0$.
    These, along with  (\ref{yu-9-28-5}) and (\ref{yu-11-10-1}), yields
\begin{equation*}
    e^{-A^\top \hat{t}}\hat{z}^*\in \mbox{ker}(B^\top)\bigcap\mbox{ker}(B^\top A^\top)=\mbox{span}\{B,AB\}^\bot
    =\{0\}.
\end{equation*}
    which contradicts to (\ref{NEW2.24,10-25}). So we have $k(\hat{t};\hat{z}^*)=1$,

\vskip 5pt

\noindent \emph{Step 2. We  complete the proof of $(ii)$.}

Since $k(\hat{t};\hat{z}^*)=1$, we have
\begin{eqnarray}\label{NEW2.38,10-25}
e^{-A^\top\hat t}\hat z^*\in \mbox{ker}(B^\top)\setminus\{0\}
\;\;\mbox{and}\;\;B^\top A^\top e^{-A^\top \hat t}\hat z^*\neq 0.
\end{eqnarray}
By \eqref{2.31,10-25}$_1$, we have
\begin{eqnarray}\label{NEW2.40,10-25}
\mbox{ker}(B^\top)=\{\lambda \hat e : \lambda\in \mathbb{R}\},
\end{eqnarray}
where $\hat e$ is a unit vector in the 1-dim subspace $\mbox{ker}(B^\top)$.
By (\ref{NEW2.38,10-25}) and \eqref{NEW2.40,10-25},  we can find $\hat{\lambda}\in\mathbb{R}\setminus\{0\}$ so that
\begin{equation}\label{yu-11-10-11}
    e^{-A^\top\hat{t}}\hat{z}^*=\hat{\lambda}\hat{e}\;\;\mbox{and}\;\;B^\top A^\top e^{-A^\top\hat{t}}\hat{z}^*=\hat{\lambda}B^\top A^\top \hat{e}\neq 0.
\end{equation}
\par
    We now claim that
    \begin{eqnarray}\label{yu-10-27-2}
\lim_{t\to \hat t^-}u^*_{x_0}(t)=\pm\frac{B^\top A^\top \hat e}
{\|B^\top A^\top \hat e\|_{\mathbb{R}^m}}.
\end{eqnarray}
    (Notice that  (\ref{yu-11-10-11})$_2$ ensures that the right hand side of \eqref{yu-10-27-2} makes sense.)
    Indeed, by (\ref{main-theorem-lemma2-1}) and  (\ref{new2.24,10-22}), we have
    \begin{eqnarray*}
    \lim_{t\to \hat t^-}u^*_{x_0}(t)=
    \lim_{t\to \hat t^-}\frac{B^\top e^{-A^\top t}\hat{z}^*}
    {\|B^\top e^{-A^\top t}\hat{z}^*\|_{\mathbb{R}^m}}.
\end{eqnarray*}
    From the above and (\ref{yu-10-19-4}) (with $k(\hat{t};\hat{z}^*)=1$) and (\ref{yu-11-10-11})$_2$, we are led to \eqref{yu-10-27-2} at once.

    Since $\hat{t}\in\mathbb{S}_{x_0}$ was arbitrarily taken, we see, from (\ref{yu-10-27-2})
    and (\ref{yu-9-27-10}),
    that $\mathbb{D}_{x_0}\subset\left\{\frac{B^\top A^\top \hat e}
{\|B^\top A^\top \hat e\|_{\mathbb{R}^m}},\frac{B^\top A^\top \hat e}
{\|B^\top A^\top \hat e\|_{\mathbb{R}^m}}\right\}$.
\par
     In summary, we finish the proof of Theorem \ref{yu-proposition-10-16-1}.
\end{proof}

\section{Appendix}\label{appendix,11-3}
\subsection{Some results on $(\widetilde{\mathcal{TP}})_{x_0}$ given by (\ref{extended-optimal-time-problem})}\label{appendix-1}
\begin{lemma}\label{extended-time-problem-lemma1}
 Suppose that $(\widetilde{\mathcal{TP}})_{x_0}$ has an admissible control. Then the following conclusion are true:
$(i)$ $0< \widetilde T^*_{x_0} <+\infty$. $(ii)$ The problem $(\widetilde{\mathcal{TP}})_{x_0}$ has an optimal control $\tilde{u}^*_{x_0}$. $(iii)$ There exists $z^*\in \mbox{span}\{B,AB,\ldots,A^{n-1}B\}$ with the property that
\begin{eqnarray*}
B^\top e^{A^\top \left(\widetilde{T}^*_{x_0}-\cdot\right) } z^* \neq 0
\;\;\mbox{in}\;\;
\mathcal{C}((0,\widetilde{T}^*_{x_0});\mathbb R^m),
\end{eqnarray*}
so that for any optimal control $\tilde{u}^*_{x_0}$ of  $(\widetilde{\mathcal{TP}})_{x_0}$,
\begin{eqnarray*}
 \left\langle \tilde{u}^*_{x_0}(t),B^\top e^{ A^\top \left(\widetilde{T}^*_{x_0}-t\right) } z^* \right\rangle_{\mathbb R^m}
 =\max_{v\in B_1(0)}   \left\langle v,B^\top e^{ A^\top \left(\widetilde{T}^*_{x_0}-t\right) } z^* \right\rangle_{\mathbb R^m}
 \;\mbox{a.e.}\;
 t\in (0,\widetilde{T}^*_{x_0}).
\end{eqnarray*}
$(iv)$ Any optimal control $\tilde{u}^*_{x_0}$ has the following expression in $(0,\widetilde{T}^*_{x_0})$:
\begin{eqnarray}\label{main-theorem-lemma2-1-b}
\tilde{u}^*_{x_0}(t)=
\begin{cases}
    \frac{B^\top e^{A^\top(\widetilde{T}^*_{x_0}- t)} z^*}{\|B^\top e^{A^\top(\widetilde{T}^*_{x_0}- t) } z^* \|_{\mathbb R^m}}&\forall\;t\in (0,\widetilde{T}^*_{x_0})\setminus(\widetilde{T}^*_{x_0}-\mathcal{O}_{z^*}),\\
    \lim_{s\to t^-} \frac{B^\top e^{ A^\top(\widetilde{T}^*_{x_0}- s)} z^*}{\|B^\top e^{A^\top(\widetilde{T}^*_{x_0}- s) } z^* \|_{\mathbb R^m}} &\forall\; t\in (\widetilde{T}^*_{x_0}-\mathcal{O}_{z^*})\bigcap(0,\widetilde T^*_{x_0}),
\end{cases}
\end{eqnarray}
    where $z^*$ is given in (iii) and $\mathcal{O}_{z^*}$ is given by (\ref{main-theorem-lemma1-1}) with $z=z^*$.
$(v)$ The optimal control to $(\widetilde{\mathcal{TP}})_{x_0}$ has the bang-bang property, i.e.,
    $\|\tilde{u}^*(t)\|_{\mathbb{R}^m}=1$ a.e. $t\in (0,\widetilde{T}^*_{x_0})$.
$(vi)$ The optimal control to $(\widetilde{\mathcal{TP}})_{x_0}$ is unique.
\end{lemma}

\begin{proof}
Let  $x_0\in\mathbb{R}^n\setminus\{0\}$ be given such that $(\widetilde{\mathcal{TP}})_{x_0}$ has an admissible control.
   The proof of the conclusion $(i)$ is standard; the conclusion $(ii)$ follows from \cite[Theorem 3.11, p113]{WWXZ};
   the conclusion  $(iv)$ can be easily obtained from $(iii)$; the conclusion  $(v)$ follows from $(iv)$ at once;
   the conclusion $(vi)$ is a consequence of $(v)$ (see, for instance, \cite{Fatt} and \cite{WWXZ}).
   For the conclusion $(iii)$, we have not found suitable reference and so give its detailed proof.
\vskip 5pt
  \noindent  \emph{Proof of $(iii)$.} For the case when
    $\mbox{rank}(B,AB,\cdots, A^{n-1}B)=n$, we can use the standard argument in \cite[Theorem 47, pp. 433]{Sontag}
    to get the desired result.

    Now we consider the case that
   $k\triangleq\mbox{rank}(B,AB,\cdots, A^{n-1}B)<n$. We  first recall  the well-known Kalman controllability decomposition
    (see \cite[Lemma 3.3.3 and Lemma 3.3.4, pp. 93]{Sontag}) which is given in the manner:
   let
    $\{e_j\}_{j=1}^n$ be an orthonormal basis of $\mathbb{R}^n$ so that
    $\{e_j\}_{j=1}^k\subset \mbox{span}\{B,AB,\ldots,A^{n-1}B\}$ and
    $\{e_j\}_{j={k+1}}^n\subset \mbox{span}\{B,AB,\ldots,A^{n-1}B\}^\bot$. Define
\begin{equation}\label{yu-9-19-b-2}
    P\triangleq(e_1,e_2,\cdots,e_n)^{\top}.
\end{equation}
     Then
\begin{equation}\label{yu-9-17-50}
    PAP^{-1}=\left(
               \begin{array}{cc}
                 A_1 & A_2 \\
                 0 & A_3 \\
               \end{array}
             \right),\;\;PB=\left(
                              \begin{array}{c}
                                B_1 \\
                                0 \\
                              \end{array}
                            \right)
\end{equation}
    with $\mbox{rank}(B_1,A_1B_1,\cdots,A_1^{k-1}B_1)=k$, where $A_1\in\mathbb{R}^{k\times k}$,
    $A_2\in\mathbb{R}^{k\times (n-k)}$, $A_3\in\mathbb{R}^{(n-k)\times (n-k)}$ and
    $B_1\in\mathbb{R}^{k\times m}$.
         For each $u\in L^\infty(\mathbb{R}^+;B_1(0))$, let
\begin{equation}\label{yu-9-17-51-bb}
    \hat{x}(\cdot;u)=\left(
                       \begin{array}{c}
                         \hat{x}_1(\cdot;u) \\
                         \hat{x}_2(\cdot;u) \\
                       \end{array}
                     \right)
    \triangleq Px(\cdot;x_0,u)\;\;\mbox{and}\;\;\hat{x}_0=\left(
                       \begin{array}{c}
                         \hat{x}_{0,1} \\
                         \hat{x}_{0,2} \\
                       \end{array}
                     \right)\triangleq Px_0.
\end{equation}
    Then, by  (\ref{system}),
    $\hat{x}(\cdot;u)$ is the solution of the following equation:
\begin{equation}\label{yu-9-17-51}
\begin{cases}
    \hat{x}_1'(t)=A_1\hat{x}_1(t)+A_2\hat{x}_2(t) +B_1u(t),\\
    \hat{x}_2'(t)=A_3\hat{x}_2(t),
\end{cases}
    t\geq 0;\;\;\;\hat{x}(0)=\hat{x}_0,
\end{equation}
    where $A_1,A_2,A_3$ and $B_1$ are given by (\ref{yu-9-17-50}).
  Several facts are given in order. Fact One: Since $(\widetilde{\mathcal{TP}})_{x_0}$ has an admissible control,
    it follows by (\ref{yu-9-17-51-bb}) and (\ref{yu-9-17-51}) that
\begin{equation*}
    \hat{x}_{0,2}=0\;\; \mbox{and}\; \hat{x}_2(\cdot;u)\equiv0\;\;\mbox{for any}\;u\in L^\infty(\mathbb{R}^+;B_1(0)).
\end{equation*}
Fact Two: By Fact One, one can directly check that the problem $(\widetilde{\mathcal{TP}})_{x_0}$ is equivalent to the following problem
$(\widehat{\mathcal{TP}})_{\hat{x}_{0,1}}$:
\begin{eqnarray*}
  \widehat T^*_{\hat{x}_{0,1}}\triangleq   \inf\big\{\hat t>0:
  \exists\, u\in L^\infty(\mathbb R^+;B_1(0))
  ~~\mbox{s.t.}~~
  \hat{x}_1(\hat t;u)=0
  \big\},
\end{eqnarray*}
i.e., they have the same optimal time and an optimal control to one problem is that of another problem.
(Notice that  $\hat{x}_1(\cdot;u)$ solves the equation: $ \hat{x}_1'(t)=A_1\hat{x}_1(t)+B_1u(t),\;t>0$; $\hat{x}_1(0)=\hat{x}_{0,1}$.)
Fact Three: Since $\mbox{rank}(B_1,A_1B_1,\cdots,A_1^{k-1}B_1)=k$, we can use the standard argument in \cite[Theorem 47, pp. 433]{Sontag} to the problem $(\widehat{\mathcal{TP}})_{\hat{x}_{0,1}}$ to find
  $\hat{z}_1^*\in\mathbb{R}^k$ with
   \begin{equation}\label{yu-9-17-52}
    B_1^\top e^{A_1^\top(\widehat{T}^*_{\hat{x}_{0,1}}-\cdot)} \hat{z}_1^*\neq 0\;\;\mbox{in}\;
    \mathcal{C}((0,\widehat{T}^*_{\hat{x}_{0,1}});\mathbb{R}^m)
\end{equation}
     so that any optimal control $\hat{u}^*_{\hat{x}_{0,1}}$ to $(\widehat{\mathcal{TP}})_{\hat{x}_{0,1}}$ satisfies
     that for a.e. $ t\in (0,\widehat{T}^*_{\hat{x}_{0,1}})$,
\begin{equation*}
 \left\langle \hat{u}^*_{\hat{x}_{0,1}}(t),B_1^\top e^{ A_1^\top \left(\widehat{T}^*_{\hat{x}_{0,1}}-t\right) } \hat{z}^*_1 \right\rangle_{\mathbb R^m}
 =\max_{v\in B_1(0)}   \left\langle v,B_1^\top e^{ A_1^\top \left(\widehat{T}^*_{\hat{x}_{0,1}}-t\right) } \hat{z}^* \right\rangle_{\mathbb R^m}.
\end{equation*}

 We now let $z^*\triangleq P^{-1}\left(
                  \hat{z}^*_1,
                  0
              \right)^\top$.
 Then we have  that     $z^*\in\mathbb{R}^n\setminus\{0\}$ (see  (\ref{yu-9-17-52}));
     $z^*\in \mbox{span}\{B,AB,\ldots,A^{n-1}B\}$ (see (\ref{yu-9-19-b-2})); and
$B^\top e^{A^\top(\widehat{T}^*_{\hat{x}_{0,1}}-\cdot)}z^*=B_1^\top e^{A_1^\top(\widehat{T}^*_{\hat{x}_{0,1}}-\cdot)}\hat{z}^*_1$.
These, along with the above Fact Two and Fact Three, leads to $(iii)$.

   Thus we end the proof of Lemma \ref{extended-time-problem-lemma1}.
  \end{proof}

\subsection{Proof of Proposition \ref{pontryagin-maximum-principle}}\label{appendix-2}
    Before presenting the proof of Proposition \ref{pontryagin-maximum-principle}, we first give two lemmas.
\begin{lemma}\label{extended-time-problem-lemma2}
  Given $x_0\in \mathbb R^n$ and $T>0$, it holds that
\begin{eqnarray}\label{extended-time-problem-lemma2-1}
\left\{x(T;x_0,u) :  u\in \mathcal{PC}((0,T);B_1(0))\right\}
=\left\{x(T;x_0,u) :  u\in L^\infty(0,T;B_1(0))\right\}.
\end{eqnarray}
\end{lemma}
\begin{proof}
    By contradiction, we suppose that \eqref{extended-time-problem-lemma2-1} is not true for some
     $x_0\in\mathbb R^n$ and $T>0$. Then since
     \begin{equation*}\label{extended-time-problem-lemma2-3}
  E_{\mathcal{PC}}\triangleq\left\{x(T;x_0,u):u\in \mathcal{PC}((0,T);B_1(0))\right\}\subset\left\{x(T;x_0,u):u\in L^{\infty}(0,T;B_1(0))\right\}
  \triangleq E_{\infty},
\end{equation*}
   we can find $\hat{u}\in L^{\infty}(0,T;B_1(0))\setminus\{0\}$ so that
\begin{equation}\label{yu-9-17-1}
    x_1\triangleq x(T;x_0,\hat{u})\in E_{\infty}\setminus E_{\mathcal{PC}}.
\end{equation}
    In what follows, we will find $\hat u^*\in \mathcal{PC}((0,T);B_1(0))$ so that
    $x(T;x_0,\hat u^*)=x_1$. When this is done, we have $x_1\in \mathcal{PC}((0,T);B_1(0))$ which contradicts
    \eqref{yu-9-17-1} and end the proof of this lemma.

    To get the above $\hat u^*$, we define the set:
    \begin{equation*}\label{yu-9-17-2}
    \mathcal{M}(T;x_0,x_1)\triangleq\{u\in L^\infty(0,T;B_1(0)): x(T;x_0,u)=x_1\}.
\end{equation*}
    It is obvious that $\mathcal{M}(T;x_0,x_1)\neq\emptyset$ (since $\hat{u}\in \mathcal{M}(T;x_0,x_1)$) and
    that $\mathcal{M}(T;x_0,x_1)$ is
    weakly star compact. We then consider the minimal norm
    control problem:
\begin{equation}\label{yu-9-17-3}
  N(x_0,x_1)\triangleq\inf\{\|u\|_{L^\infty(0,T;B_1(0))}:u\in\mathcal{M}(T;x_0,x_1)\}.
  \end{equation}
     We can easily check that  the problem \eqref{yu-9-17-3} has a solution $v^*$.
      We also have $N(x_0,x_1)\neq 0$, for otherwise, the function: $v^*\equiv 0$ (over $(0,T)$) belongs to
      $\mathcal{PC}((0,T);B_1(0))$ and satisfies that $x(T;x_0,v^*)=x_1$ which contradicts \eqref{yu-9-17-1}.
      So we have $1\geq N(x_0,x_1)>0$. Thus, for each $\varepsilon\in(0,1)$,
     \begin{equation}\label{yu-9-17-4}
    x_1\notin\{x(T;x_0,u):u\in L^\infty(0,T;B_{(1-\varepsilon)N(x_0,x_1)}(0))\} \triangleq
    \mathcal{Q}(B_{(1-\varepsilon)N(x_0,x_1)}(0);x_0).
\end{equation}
(Indeed, if  $x_1\in \mathcal{Q}(B_{(1-\varepsilon)N(x_0,x_1)}(0);x_0)$ for some $\varepsilon\in(0,1)$, then there is
$v_\varepsilon$ with $\|v_\varepsilon\|_{L^\infty(0,T;\mathbb{R}^m)}\leq (1-\varepsilon)N(x_0,x_1)$ so that
$x(T;x_0,v_\varepsilon)=x_1$. This yields $v_\varepsilon\in \mathcal{M}(T;x_0,x_1)$, which, together with \eqref{yu-9-17-3},
leads to the contradiction: $0<N(x_0,x_1)\leq \|v_\varepsilon\|_{L^\infty(0,T;\mathbb{R}^m)}\leq (1-\varepsilon)N(x_0,x_1)$.)
        Note that  $\mathcal{Q}(B_{(1-\varepsilon)N(x_0,x_1)}(0);x_0)$ is convex, closed and bounded.
   Thus by \eqref{yu-9-17-4}, we can apply  the Hahn-Banach theorem to find, for each  $\varepsilon\in(0,1)$, $z_\varepsilon\in\mathbb{R}^n$ with $\|z_\varepsilon\|_{\mathbb{R}^n}=1$ so that
\begin{equation}\label{yu-9-17-5}
    \langle x_1,z_\varepsilon\rangle_{\mathbb{R}^n} \geq \langle x(T;x_0,v),z_\varepsilon\rangle_{\mathbb{R}^n}\;\;\mbox{for any}\;v\in L^\infty(0,T;B_{(1-\varepsilon)N(x_0,x_1)}(0)).
\end{equation}

  Given  $u\in L^\infty(0,T;B_1(0))$ and $\varepsilon\in(0,1)$, we have that
$$
  v_\varepsilon\triangleq (1-\varepsilon)N(x_0,x_1)u\in L^\infty(0,T;B_{(1-\varepsilon)N(x_0,x_1)}(0)).
$$
       This, together with (\ref{yu-9-17-5}) and the fact that $x(T;x_0,v^*)=x_1$, yields that for any $u\in L^\infty(0,T;B_1(0))$,
\begin{equation}\label{yu-9-17-7}
    \int_0^T\langle B^\top e^{A^\top(T-t)}z_\varepsilon,v^*(t)\rangle_{\mathbb{R}^m}dt
    \geq (1-\varepsilon)N(x_0,x_1)\int_0^T\langle B^\top e^{A^\top(T-t)}z_\varepsilon,u(t)\rangle_{\mathbb{R}^m}dt.
\end{equation}
   Since $\|z_\varepsilon\|_{\mathbb{R}^n}=1$, there is a sequence
    $\{\varepsilon_j\}_{j\in\mathbb{N}}\subset(0,1)$ and $z^*\in\mathbb{R}^n$ with $\|z^*\|_{\mathbb{R}^n}=1$ such that $\varepsilon_j\to 0$ and $z_{\varepsilon_j}\to z^*$ as $j\to+\infty$.
    By taking
    $\varepsilon=\varepsilon_j$ in (\ref{yu-9-17-7}) and then sending $j\to+\infty$, we get that  for any $u\in L^\infty(0,T;B_1(0))$,
\begin{equation*}\label{yu-9-17-8}
    \int_0^T\langle B^\top e^{A^\top(T-t)}z^*,v^*(t)\rangle_{\mathbb{R}^m}dt
    \geq N(x_0,x_1)\int_0^T\langle B^\top e^{A^\top(T-t)}z^*,u(t)\rangle_{\mathbb{R}^m}dt.
\end{equation*}
    Using a standard argument involving the density of Lebesgue points in the above, we obtain
\begin{equation*}\label{yu-9-17-9}
    \langle B^\top e^{A^\top(T-t)}z^*,v^*(t)\rangle_{\mathbb{R}^m}
    =N(x_0,x_1)\max_{u\in B_1(0)}\langle B^\top e^{A^\top(T-t)}z^*,u\rangle_{\mathbb{R}^m}\;\;\mbox{for a.e.}\;t\in(0,T),
\end{equation*}
   from which, it follows that
\begin{equation*}\label{yu-9-17-10}
    v^*(t)=
    N(x_0,x_1)\frac{B^\top e^{A^\top(T-t)}z^*}{\|B^\top e^{A^\top(T-t)}z^*\|_{\mathbb{R}^m}}\;\;\mbox{for a.e.}\;
    t\in(0,T)\setminus (T-\mathcal{O}_{z^*}),
\end{equation*}
    where $\mathcal{O}_{z^*}$ is given by (\ref{main-theorem-lemma1-1}) with $z=z^*$.

  We now set
    \begin{equation*}
    \hat{u}^*(t)\triangleq
\begin{cases}
    N(x_0,x_1)\frac{B^\top e^{A^\top(T-t)}z^*}{\|B^\top e^{A^\top(T-t)}z^*\|_{\mathbb{R}^m}} &\mbox{if}\;t\in(0,T)\setminus (T-\mathcal{O}_{z^*}),\\
    N(x_0,x_1)\lim_{s\to t^-}\frac{B^\top e^{A^\top(T-s)}z^*}{\|B^\top e^{A^\top(T-s)}z^*\|_{\mathbb{R}^m}}&\mbox{if}\;t\in(T-\mathcal{O}_{z^*})\bigcap(0,T).
\end{cases}
\end{equation*}
    Since $\mathcal{O}_{z^*}\cap (0,T)$  contains at most finitely many elements (see Lemma \ref{main-theorem-lemma1} with $z=z^*$), we have
    that
    $$
    \hat{u}^*\in \mathcal{PC}((0,T);B_{N(x_0,x_1)}(0))
    \subset \mathcal{PC}((0,T);B_1(0))
    $$
    and that $\hat u^*(t)=v^*(t)$ for a.e. $t\in (0,T)$. The later yields that
    $x(T;x_0,\hat{u}^*)=x_1$. So $\hat u^*$ is desired.

    Hence, we finish the proof of Lemma \ref{extended-time-problem-lemma2}.
\end{proof}
\begin{lemma}\label{yu-lemma-10-11-1}
   Suppose that  $(\textbf{H})_{x_0}$ is true. Then
  $(\mathcal{TP})_{x_0}$ and $(\widetilde{\mathcal{TP}})_{x_0}$ are equivalent.
\end{lemma}
\begin{proof}
    We divide the proof into several steps.
\vskip 5pt
  \noindent  \emph{Step 1. We show that $
 T^*_{x_0}=\widetilde T^*_{x_0}<+\infty$,
   where $\widetilde T^*_{x_0}$ is the optimal time of $(\widetilde{\mathcal{TP}})_{x_0}$ defined by (\ref{extended-optimal-time-problem}).}
\par
    By the assumption $(\textbf{H})_{x_0}$, we can use a standard way to get
$T^*_{x_0}<+\infty$.
Meanwhile, since $\mathcal{PC}(\mathbb{R}^+;B_1(0))\subset L^{\infty}(\mathbb{R}^+;B_1(0))$,
we have $\widetilde{T}^*_{x_0}\leq T^*_{x_0}$.

Now by  contradiction, we suppose that the conclusions in {\it Step 1} is not true.
Then we have
$\widetilde{T}^*_{x_0}<T^*_{x_0}$.
   From $(ii)$ in Lemma \ref{extended-time-problem-lemma1}, we can find
$\tilde{u}^*_{x_0}\in L^{\infty}(\mathbb{R}^+; B_1(0))$ so that
  $x(\widetilde{T}^*_{x_0};x_0,\tilde{u}^*_{x_0})=0$. Then by Lemma \ref{extended-time-problem-lemma2}, there is  $\hat u\in \mathcal{PC}(\mathbb{R}^+; B_1(0))$ so that
  $x(\widetilde{T}^*_{x_0};x_0,\hat u)=x(\widetilde{T}^*_{x_0};x_0,\tilde{u}^*_{x_0})=0$,
which, together with the optimality of  $T^*_{x_0}$ (see (\ref{time-optimal-problem})), indicates that
$T^*_{x_0}\leq \widetilde{T}^*_{x_0}$. This leads to a contradiction. So conclusions in {\it Step 1} are true.

\vskip 5pt
 \noindent   \emph{Step 2. We prove that if $\tilde{u}^*_{x_0}$ is an optimal control to $(\widetilde{\mathcal{TP}})_{x_0}$, then there is $v^*\in \mathcal{PC}(\mathbb{R}^+;B_1(0))$, with $v^*(t)=\tilde{u}^*_{x_0}(t)$ for a.e. $t\in \mathbb{R}^+$,
 so that $v^*$ is an optimal control to $(\mathcal{TP})_{x_0}$.}

\par
   By $(vi)$ of Lemma \ref{extended-time-problem-lemma1},
   $\tilde{u}^*_{x_0}|_{(0,\widetilde{T}^*_{x_0})}$  has the expression
    (\ref{main-theorem-lemma2-1-b}). By the definition of an optimal control to $(\widetilde{\mathcal{TP}})_{x_0}$
   (see Section \ref{introduce}), we have $\tilde{u}^*_{x_0}(t)=0$ for a.e. $t\in (\widetilde{T}^*_{x_0},+\infty)$. We now define a function
   $v^*: \mathbb{R}^+\rightarrow\mathbb{R}^m$ in the manner:
   $v^*(t)\triangleq\tilde{u}^*_{x_0}|_{(0,\widetilde{T}^*_{x_0})}(t)$ for each $t\in (0, \widetilde{T}^*_{x_0})$;
   $v^*(\widetilde{T}^*_{x_0})\triangleq\lim_{s\to \widetilde{T}_{x_0}^{*-}}\tilde{u}^*_{x_0}(s)$;
   $v^*(t)\triangleq 0$ for each $t\in (\widetilde{T}^*_{x_0},+\infty)$. Then we have that
   $v^*=\tilde{u}^*_{x_0}$ in $L^\infty(\mathbb{R}^+;\mathbb{R}^m)$; $v^*\in \mathcal{PC}(\mathbb{R}^+;B_1(0))$; and
   $x(T^*_{x_0};x_0,v^*)=0$ (Here, we used the above  {\it Step 1}.) These imply that $v^*$ is an optimal control to $(\mathcal{TP})_{x_0}$.

 \vskip 5pt

 \noindent {\it Step 3. It is clear that any optimal control to $(\mathcal{TP})_{x_0}$ is an optimal control to $(\widetilde{\mathcal{TP}})_{x_0}$, since $\mathcal{PC}(\mathbb{R}^+;B_1(0))\subset L^\infty(\mathbb{R}^+;B_1(0))$
 and $T^*_{x_0}=\widetilde T^*_{x_0}$.}

Hence, we end the proof of    Lemma \ref{yu-lemma-10-11-1}.
\end{proof}

    \vskip 5pt

\begin{proof}[Proof of Proposition \ref{pontryagin-maximum-principle}]

    The conclusions in Proposition \ref{pontryagin-maximum-principle} follow from
 Lemma \ref{extended-time-problem-lemma1} and Lemma \ref{yu-lemma-10-11-1}.
\end{proof}

\vskip 10pt

\textbf{Acknowledgments.} The authors thank Prof. Andrei A. Agrachev for introducing the book \cite{Agrachev-Sachkov} and Ph.D Thesis \cite{Biolo}.


\begin{thebibliography}{1}
\bibitem{Agrachev-Biolo-2017}  A. A. Agrachev and C. Biolo, \textit{Switching in time-optimal problem: the 3D Case with 2D control}, J. Dyn. Control Syst., 23 (2017), 577-595.

\bibitem{Agrachev-Biolo-2018} A. A. Agrachev and C. Biolo, \textit{Switching in time-optimal problem in a ball}, SIAM J. Control Optim., 56(1) (2018), 183-120.
\bibitem{Agrachev-Sachkov} A. A. Agrachev and Y. L. Sachkov, Control theory from the geometric viewpoint, Springer-Verlag, 2004.
\bibitem{Bellman} R. Bellman, I. Glicksberg  and O. Gross,  \textit{On the ¡°bang-bang¡± control problem}. Quarterly of Applied Mathematics, 14(1) (1956), 11-18.

\bibitem{Biolo} C. Biolo, \textit{Switching in time-optimal problem}, Scuola Internazionale Superiore di Studi Avanzati - Trieste, Area of Mathematics, Ph.D. in Mathematical Analysis, Modelling and Application
Thesis, 2017.
\bibitem{Conti} R. Conti, Teoia del Controllo e del Controllo Ottimo, UTET, Torino, Italy, 1974.
\bibitem{Evans} L. C. Evans, An Introduction to Mathematical Optimal Control Theory. Lecture Notes, Univerisity of California, Department of Mathematics, Berkeley, 2005.

\bibitem {Fatt} H. O. Fattorini, Infinite Dimensional Linear Control Systems. The Time Optimal and Norm
Optimal Control Problems, North-Holland Math. Stud. 201, Elsevier, Amsterdam, 2005.

\bibitem {Fatt-2011} H. O. Fattorini, \textit{Time and norm optimal controls: a survey of recent results and open problems}. Acta Mathematica Scientia 31.6 (2011), 2203-2218.

\bibitem{LaSalle} J. P. LaSalle, \textit{Time optimal control systems}. Proceedings of the National Academy of Sciences of the United States of America 45(4) (1959), 573-577.
\bibitem{Lin-Wang} P. Lin and G. Wang, \textit{Blowup time optimal control for ordinary differential equations}, SIAM J. Control Optim. 49(1) (2011), 73-105.
\bibitem{Poggiolini-2017} L. Poggiolini, \textit{Structural stability of bang-bang trajectories with a double switching time in the minimum time problem}, SIAM J. Control Optim., 55(6) (2017), 3779-3798.

\bibitem{Pontryagin} L. S. Pontryagin, V. G. Boltyanski and
R.V. Gamkrelidze, et el, Mathematical Theory of Optimal Processes, New York, Wiley, 1962.

\bibitem{QSL2} S. Qin and G. Wang, \textit{Controllability of impulse controlled systems of heat equations coupled by constant matrices}, J. Differential Equations, 263 (2017), 6456-6493.


\bibitem{Sontag} E. D. Sontag, Mathematical Control Theory: Deterministic Finite-Dimensional Systems, 2nd edition, Springer-Verlag, New York, 1998.

\bibitem{Sussmann-1979} H. J. Sussmann, \textit{A bang-bang theorem with bounds on the number of swichtings},
SIAM J. Control and Optim., 17(5) (1979), 629-651.


\bibitem{WWXZ} G. Wang, L. Wang, Y. Xu and Y. Zhang. Time Optimal Control of Evolution Equations. Progress in Nonlinear
Dierential Equations and their Applications, 92. Subseries in Control. Birkh\"{a}user/Springer, Cham, 2018.

\bibitem{Wang-Zhang} G. Wang, Y. Zhang, \textit{Decompositions and bang-bang properties}. Math. Control Relat. Fields, 7(1) (2017),  73-170.














\end{thebibliography}
 \end{document}